\documentclass[a4paper,11pt,english]{article}

\usepackage[utf8]{inputenc}
\usepackage[T1]{fontenc}
\usepackage{babel}

\usepackage{dirtytalk}

\usepackage{pstricks}

\usepackage{fullpage}

\usepackage{amsfonts,dsfont,pifont,amsmath,amsthm,amssymb,amsbsy,mathabx,upgreek}

\usepackage{graphicx,color}
\usepackage{multirow}

\usepackage{units}
\usepackage{bigints}
\usepackage{csquotes}

\usepackage{url}

\usepackage{cite}

\usepackage[pdftex,colorlinks=true,pagebackref=true
,linkcolor=blue,bookmarks=true,bookmarksopen=true,citecolor=blue]{hyperref}

\usepackage{titletoc}

\usepackage{mathptmx}

\newtheorem{theo}{Theorem}[section]

\newtheorem{lem}{Lemma}[section]
\newtheorem{prop}{Proposition}[section]

\newtheorem{rem}{Remark}[section]

\newtheorem{ass}{Assumption}[section]

\newcommand{\Nset}{\mathbb{N}}

\newcommand{\E}{\mbox{$\mathbb{E}$}}

\newcommand{\PP}{\mbox{$\mathbb{P}$}}

\newcommand{\CA}{\mathcal{A}}

\newcommand{\ep}{\mbox{$\varepsilon$}}

\newcommand{\scr}[1]{\scriptscriptstyle #1}

\newcommand{\lpref}{\smash{\raisebox{3.5pt}{\!\!\!\begin{tabular}{c}$\hskip-4pt\scriptstyle\longleftarrow$ \\[-7pt]{\rm pref}\end{tabular}\!\!}}}
\newcommand{\petitlpref}{\smash{\raisebox{2.5pt}{\!\!\!\begin{tabular}{c}$\hskip-2pt\scriptscriptstyle\longleftarrow$ \\[-9pt]{$\scriptstyle\hskip 1pt\rm pref$}\end{tabular}\!\!}}}

\newcommand{\ttu}{\mathtt u}
\newcommand{\ttd}{\mathtt d}

\newcommand{\tail}{\mathcal{T}}

\makeatletter
\newcommand{\xRightarrow}[2][]{\ext@arrow 0359\Rightarrowfill@{#1}{#2}}
\makeatother
{}

\makeatletter
\newcommand{\xsim}[2][]{\ext@arrow 0359\Simfill@{#1}{#2}}
\makeatother

\usepackage[draft]{fixme}
\fxusetheme{color}
\FXRegisterAuthor{a}{aa}{A}
\FXRegisterAuthor{pe}{ape}{P}
\FXRegisterAuthor{yo}{ayo}{Yo}
\FXRegisterAuthor{b}{ab}{B}

\usepackage[blocks]{authblk}

\author[1]{{\bf Peggy Cénac}}
\affil[1]{Institut de Mathématiques de Bourgogne (IMB) - UMR CNRS 5584

Université de Bourgogne Franche-Comté, 21000 Dijon, France}

\author[2]{{\bf Arnaud Le Ny}}
\affil[2]{Laboratoire d'Analyse et de Mathématiques Appliquées (LAMA) - UMR CNRS 8050

Université Paris Est Créteil, 94010 Créteil Cedex, France}

\author[3]{{\bf Basile de Loynes}}
\affil[3]{Ensai - Université de Bretagne-Loire, Campus de Ker-Lann, Rue Blaise Pascal, BP 37203, 35172 BRUZ cedex, France}

\author[1]{{\bf Yoann Offret}}

\begin{document}

\title{\bf Persistent random walks. II. Functional Scaling Limits}

\date{}

\maketitle

\noindent
\rule{\linewidth}{.5pt}

\vspace{1em}

\noindent
{\small {\bf Abstract}\; We give a complete and unified description -- under some stability assumptions -- of the functional scaling limits associated with some persistent random walks for which the recurrent or transient type is studied in \cite{PRWI}. As a result, we highlight a phase transition phenomenon with respect to the memory. It turns out that the limit process is either Markovian or not according to -- to put it in a nutshell --  the rate of decrease of the distribution tails corresponding to the persistent times. In the memoryless situation, the limits are classical strictly stable Lévy processes of infinite variations. However, we point out that the description of the critical Cauchy case fills some lacuna even in the closely related context of Directionally Reinforced Random Walks (DRRWs) for which it has not been considered yet. Besides, we need to introduced some relevant generalized drift -- extended the classical one --  in order to study the critical case but also the situation when the limit is no longer Markovian. It  appears  to be  in full generality a drift in mean for the Persistent Random Walk (PRW). The limit processes keeping some memory -- given by some variable length Markov chain -- of the underlying PRW  are called arcsine Lamperti anomalous diffusions due to their marginal distribution which are computed explicitly here. To this end, we make the connection with the governing equations for Lévy walks, the occupation times of skew Bessel processes and a more general class modelled on Lamperti processes.  We also stress that we clarify some misunderstanding regarding this marginal distribution in the framework of DRRWs. Finally, we stress that the latter situation is more flexible -- as in the first paper -- in the sense that the results can be easily generalized to a wider class of PRWs without renewal pattern.}

\vspace{1em}


\vspace{1em}


\tableofcontents

%%%%%%%%%%%%%%%%%%%%%%%%%%%%%%%%%%%%%%%%%%%%%%%%%%%%%%%%%%%%%%%%%%%%%%%%%%%%%%%%%%%%%%%%%%%%%%%%%%%%%%%%%%%%%%%%%%%%%%%%%%%%%%%%%%%%%%%%%%%%%%%%%%%%%%%%%%%%%%%%%%%%%%%%%%%%%%%%%%%%%%%%%%%%%%%%%%%%%%%%%%%%%%%%%%%%%%%%%%%%%%%%%%%%%%%%%%%%%%%%%%%%%%%%%%%%%%%%%%%%%%%%%%%%%%%%%%%%%%%%%%%%%%%%%%%%%%%%%%%%%%%%%%%%%%%%%%%%%%%%%%%%%%%%%%%%%%%%%%%%%%%%%%%%%%%%%%%%%%%%%%%%%%%%%%%%%%%%%%%%%%%%%%%%%%%%%%%%%%%%%%%%%%%%%%%%%%%%%%%%%%%%%%%%%%%%%%%%%%%%%%%%%%%%%%%%%%%%%%%%%%%%%%%%%%%%%%%%%%%%%%%%%%%%%%%%%%%%%%%%%%%%%%%%%%%%%%%%%%%%%%%%%%%%%%%%%%%%%%%%%%%%%%%%%%%%%%%%%%%%%%%%%%%%%%%%%%%%%%%%%%%%%%%%%%%%%%%%%%%%%%%%%%%%%%%%%%%%%%%%%%%%%%%%%%%%%%%%%%%%%%%%%%%%%%%%%%%%%%%%%%%%%%%%%%%%%%%%%%%%%%%%%%%%%%%%%%%%%%%%%%%%%%%%%%%%%%%%%%%%%%%%%%%%%%%%%%%%%%%%%%%%%%%%%%%%%%%%%%%%%%%%%%%%%%%%%%%%%%%%%%%%%%%%%%%%%%%%%%%%%%%%%%%%%%%%%%%%%%%%%%%%%%%%%%%%%%%%%%%%%%%%%%%%%%%%%%%%%%%%%%%%%%%%%%%%%%%%%%%%%%%%%%%%%%%%%%%%%%%%%%%%%%%%%%%%%%%%%%%%%%%%%%%%%%%%%%%%%%%%%%%%%%%%%%%%%%%%%%%%%%%%%%%%%%%%%%%%%%%%%%%%%%%%%%%%%%%%%%%%%%%%%%%%%%%

\section{Introduction}

\label{Introduction}

This paper is a continuation of \cite{PRWI} in which recurrence versus transience features of some PRWs are described. More specifically, we still consider a walker $\{S_{n}\}_{n\geq 0}$ on $\mathbb Z$, whose jumps are of unit size, and such that at each step it keeps the same direction (or switches) with a probability directly depending on the time already spent in the direction the walker is currently moving. Here we aim at investigating functional scaling limits of the form 
\begin{equation}\label{scalingforme}
\left\{\frac{S_{\lfloor u\,t\rfloor}-\mathbf m_{\scr S}\,u\,t}{\lambda(u)}\right\}_{t\geq 0}\quad\mbox{or}\quad\left\{\frac{S_{u\,t}-\mathbf m_{\scr S}\,u\,t}{\lambda(u)}\right\}_{t\geq 0}\;\xRightarrow[u\to\infty]{\mathcal L}\; \{Z(t)\}_{t\geq 0},	
\end{equation}
for which some  functional convergence in distribution toward a stochastic process $Z$ holds. The continuous time stochastic process $\{S_{t}\}_{t\geq 0}$ above denotes the piecewise linear interpolation of the discrete time one $\{S_{n}\}_{n\geq 0}$. Due to the sizes of its jumps, the latter is obviously ballistic or sub-ballistic. In particular, the drift parameter $\mathbf m_{\scr S}$  belongs to $[-1,1]$ and that the growth rate of the normalizing positive function $\lambda(u)$ is at most linear.  In full generality, we aim at investigating PRWs  given by
\begin{equation}\label{defpersistpart}
S_0=0\quad\mbox{and}\quad S_n:=\displaystyle\sum_{k=1}^n X_k,\quad\mbox{for all $n\geq 1$},
\end{equation}
where a two-sided process of jumps $\{X_k\}_{n\in\mathbb Z}$ in an additive group $G$ is considered. In order to take into account possibly infinite reinforcements, the increment process is supposed to have a finite but possibly unbounded variable memory. More precisely, we assume  that it is built from a Variable Length Markov Chain (VLMC) given by some probabilized context tree. To be more explicit, let us give the general construction that fits to our model. 

\subsection{VLMC structure of increments}

Let $\mathcal{L}= \CA^{-\Nset}$ be the set of left-infinite words on the alphabet $\CA:=\{ \ttd,\ttu \}\simeq \{-1,1\}$ and  consider a complete tree on this alphabet, \emph{i.e.}\@  such that each node has $0$ or $2$ children, whose leaves $\mathcal{C}$ are words (possibly infinite) on $\mathcal A$. To each leaf $c\in\mathcal C$, called a context, is attached a probability distribution $q_{c}$ on $\mathcal A$. Endowed with this probabilistic structure, such a tree is named a probabilized context tree. The related VLMC is the Markov Chain $\{U_{n}\}_{n\geq 0}$  on $\mathcal{L}$ whose transitions are given by
\begin{equation}
 \label{eq:def:VLMC}
 \PP(U_{n+1} = U_n\ell| U_n)=q_{\petitlpref (U_n)}(\ell),
 \end{equation}
where $\lpref(w)\in\mathcal C$ is defined as the shortest prefix of $w=\cdots w_{-1}w_{0}$ -- read from right to left -- appearing as a leaf of the context tree. The kth increment $X_k$ of the corresponding PRW is  given as the rightmost letter of $U_k:=\cdots X_{k-1}X_k$ with the one-to-one  correspondence $\ttd=-1$ (for a descent) and $\ttu=1$ (for a rise). Supposing the context tree is infinite, the resulting PRW  is no longer Markovian and somehow very persistent. The associated two-sided process of jumps has a finite but possibly unbounded variable memory whose successive lengths are given by the so-called age time process defined for any $n\geq 0$ by 
\begin{equation}\label{agetimediscret}
A_{n}:=\inf\{k\geq 1 : X_{n}\cdots X_{n-k}\in \mathcal C\}.
\end{equation}

\subsection{Outline of the article}

The paper is organized as follows. In the next section, we recall the elementary assumption on $S$ -- the double-infinite comb PRW -- and we introduce our main Assumption \ref{assglobal} together with the quantity $\mathbf m_{\scr S}$ in (\ref{scalingforme}) -- named the mean drift -- and the normalizing function $\lambda(u)$. Thereafter -- in Section \ref{sectionresults} -- we state our main results, namely Theorems \ref{thmgene0}, \ref{scalingbrownian} and \ref{complement}. Together, the two latter are refinements and complements of the first one.  In Section \ref{sectioncontri} we make some remarks to step back on these results. The two following Sections \ref{sectionl\'{e}vy}  and \ref{sectionanomalous} are focused on the proofs in the two fundamental cases. Finally, in the last Section \ref{appendix} we state two useful and straightforward lemmas allowing several interpretations of the Assumption \ref{assglobal} and often implicitly used in the proofs.

%%%%%%%%%%%%%%%%%%%%%%%%%%%%%%%%%%%%%%%%%%%%%%%%%%%%%%%%%%%%%%%%%%%%%%%%%%%%%%%%%%%%%%%%%%%%%%%%%%%%%%%%%%%%%%%%%%%%%%%%%%%%%%%%%%%%%%%%%%%%%%%%%%%%%%%%%%%%%%%%%%%%%%%%%%%%%%%%%%%%%%%%%%%%%%%%%%%%%%%%%%%%%%%%%%%%%%%%%%%%%%%%%%%%%%%%%%%%%%%%%%%%%%%%%%%%%%%%%%%%%%%%%%%%%%%%%%%%%%%%%%%%%%%%%%%%%%%%%%%%%%%%%%%%%%%%%%%%%%%%%%%%%%%%%%%%%%%%%%%%%%%%%%%%%%%%%%%%%%%%%%%%%%%%%%%%
%%%%%%%%%%%%%%%%%%%%%%%%%%%%%%%%%%%%%%%%%%%%%%%%%%%%%%%%%%%%%%%%%%%%%%%%%%%%%%%%%%%%%%%%%%%%%%%%%%%%%%%%%%%%%%%%%%%%%%%%%%%%%%%%%%%%%%%%%%%%%%%%%%%%%%%%%%%%%%%%%%%%%%%%%%%%%%%%%%%%%%%%%%%%%%%%%%%%%%%%%%%%%%%%%%%%%%%%%%%%%%%%%%%%%%%%%%%%%%%%%%%%%%%%%%%%%%%%%%%%%%%%%%%%%%%%%%%%%%%%%%%%%%%%%%%%%%%%%%%%%%%%%%%%%%%%%%%%%%%%%%%%%%%%%%%%%%%%%%%

\section{Framework and assumptions}

\setcounter{equation}{0}

\label{Framework}

The model we consider corresponds to the double-infinite comb. Roughly speaking, the leaves -- coding the memory -- are words on $\{\ttd,\ttu\}\simeq\{-1,1\}$ of the form $\ttd^{n}\ttu$ and $\ttu^{n}\ttd$. It follows that the probability to invert the current direction depends only on the direction itself and of its present length. In the sequel, we refer to  Figure \ref{marche} that illustrates our notations and assumptions on the so-called  double-infinite comb PRW.

\begin{figure}[!b]
\centering
\includegraphics[width=124mm]{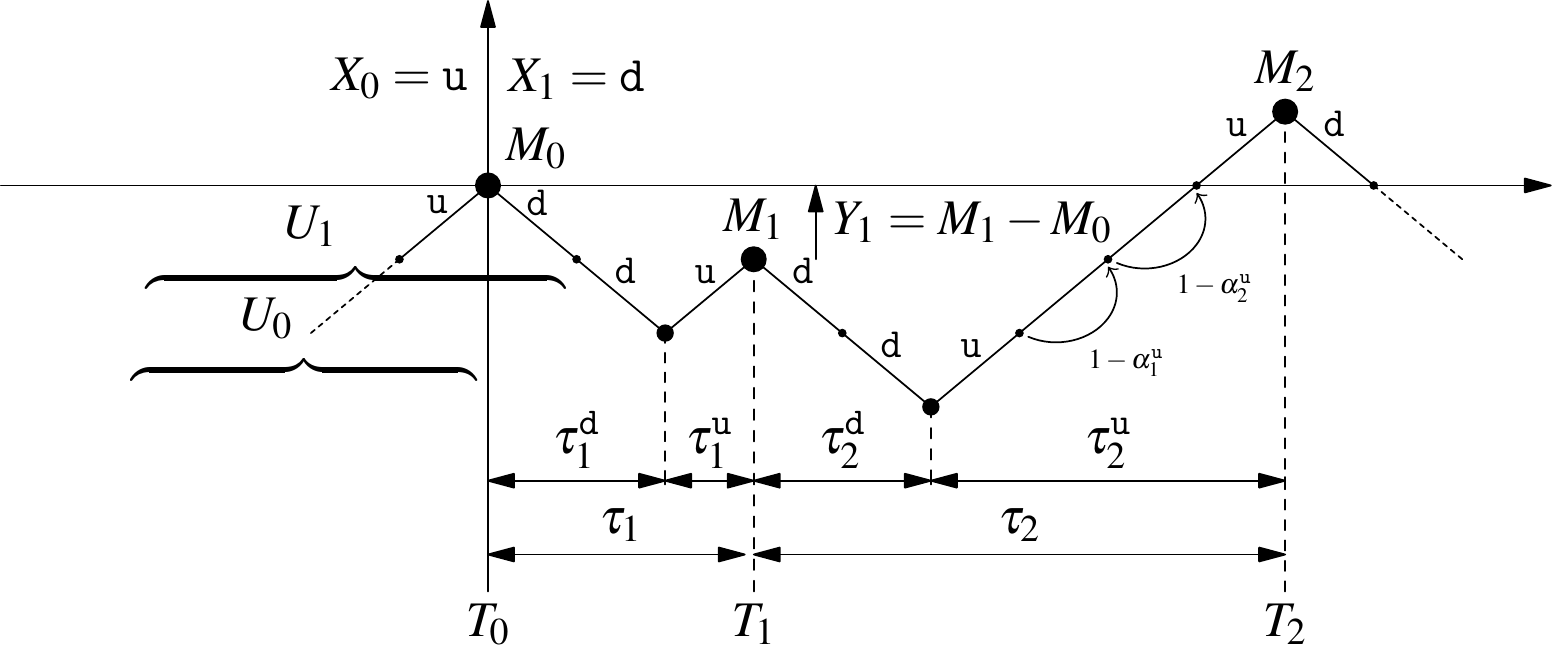}
\caption{A trajectory of S}
\label{marche}
\end{figure}

%%%%%%%%%%%%%%%%%%%%%%%%%%%%%%%%%%%%%%%%%%%%%%%%%%%%%%%%%%%%%%%%%%%%%%%%%%%%%%%%%%%%%%%%%%%%%%%%%%%%%%%%%%%%%%%%%%%%%%%%%%%%%%%%%%%%%%%%%%%%%%%%%%%%%%%%%%%%%%%%%%%%%%%%%%%%%%%%%%%%%
%%%%%%%%%%%%%%%%%%%%%%%%%%%%%%%%%%%%%%%%%%%%%%%%%%%%%%%%%%%%%%%%%%%%%%%%%%%%%%%%%%%%%%%%%%%%%%%%%%%%%%%%%%%%%%%%%%%%%%%%%%%%%%%%%%%%%%%%%%%%%%%%%%%%%%%%%%%%%%%%%%%%%%%%%%%%%%%%%%%%%%%%%%%%%%%%%%%%%%%%%%%%%%%%%%%%%%%%%%%%%%%%%%%%%%%%%%%%%%%%%%%%%%%%%%%%%%%%%%%%%%%%%%%%%%%%%%%%%%%%%%%%%%%%%%%%%%%%%%%%%%%%%%%%%%%%%%%%%%%%%%%%%%%%%%%%%%%%%%%

\subsection{Elementary settings and hypothesis}

\label{Settings}
We recall that this process  is characterized by the transition probabilities  
\begin{equation}\label{transitions}
\alpha_k^\ttd:=\PP(X_{k+1}=\ttu|X_k\cdots X_{0}=\ttd^{k}\ttu) \quad\textrm{and} \quad \alpha_k^\ttu:=\PP(X_{k+1}=\ttd|X_k\cdots X_{0}=\ttu^{k}\ttd).
\end{equation} 
where $X_{k}$ is the kth jump in $\{\ttd,\ttu\}\simeq\{-1,1\}$ given as the rightmost letter of the left-infinite word $U_{k}$ -- the kth term of underlying double-infinite comb VLMC defined in \cite{peggy}. The latter conditional probabilities are invariant by shifting the sequence of increments and thus $\alpha_k^\ttu$ and  $\alpha_k^\ttd$ stand respectively for the probabilities of changing direction after $k$ rises and $k$ descents. Therefore, we can write
\begin{equation}
\alpha_k^\ttu=\mathbb P\left(S_{n+1}-S_n=1
| X_n=\ttu,\,A_n=k\right)=1-\mathbb P\left(S_{n+1}-S_n=-1
| X_n=\ttu,\,A_n=k\right),
\end{equation} 
and vice versa replacing the direction $\ttu$ by $\ttd$ and inverting the jumps $1$ and $-1$.

Besides, in order to avoid trivial cases, we assume that $S$ can not be frozen in one of the two directions with a positive probability so that it makes infinitely many U-turns almost surely. We  deal throughout this paper with the conditional probability with respect to the event $(X_0,X_1)=(\ttu,\ttd)$. In other words, the initial time is suppose to be an up-to-down turn. Obviously, there is no loss of generality supposing this and the long time behaviour of $S$ is not affected as well. Therefore, we  assume the following
\begin{ass}[finiteness of the length of runs] \label{a0}
For any $\ell\in\{\ttu,\ttd\}$, 
\begin{equation}\label{a1}
\prod_{k=1}^{\infty}(1-\alpha_{k}^\ell)=0\;\Longleftrightarrow\;\left(\exists\,k\geq 1\;\mbox{ s.t. } \; \alpha_{k}^{\ell}=1\quad\mbox{or}\quad \sum_{k=1}^{\infty} \alpha_{k}^{\ell}=\infty\right).
\end{equation}
\end{ass}

In addition, we exclude implicitly the situation when both of the length of runs are almost surely constant. We denote by $\tau^\ttu_{n}$ and  $\tau^\ttd_{n}$ the length of the nth rise and descent respectively. The two sequences of {\it i.i.d.\@} random variables  $\{\tau_n^\ttd\}_{n\geq 1}$ and $\{\tau_n^\ttu\}_{n\geq 1}$ are independent -- a renewal property somehow --  and it is clear that their distribution tails and truncated means are given for any $\ell\in\{\ttu,\ttd\}$ and $t\geq 0$ by 
\begin{equation}
\label{def-tail}
\mathcal T_{\ell}(t):=\PP(\tau^\ell_{1} >t)=\prod_{k=1}^{\lfloor t\rfloor }(1-\alpha_{k}^\ell)\quad\mbox{and}\quad 
\Theta_{\ell}(t):=\mathbb E[\tau^{\ell}_{1}\wedge t]=\sum_{n=1}^{\lfloor t\rfloor }\prod_{k=1}^{n-1}(1-\alpha_{k}^{\ell}).	
\end{equation}
%With the latter notations, the expectations of the so-called running times or persistence times are given by $\Theta_{\ell}(\infty)$ -- the limits of $\Theta_{\ell}(t)$ as $t$ tends to infinity. 
We will also need to consider their  truncated second moments  defined by
\begin{equation}\label{truncatedmeansum}
{V}_{\ell}(t):=\mathbb E[(\tau^{\ell}_1)^{2}\mathds 1_{\{|\tau^{\ell}_1|\leq t\}}].
\end{equation}

Furthermore, in order to deal with a more tractable random walk with \emph{i.i.d.\@} increments, we introduce the underlying skeleton random walk $\{M_n\}_{n\geq 0}$  associated with the even U-turns -- the original walk observed at the random times of up-to-down turns. These observation times form also a  random walk (increasing) which is denoted by $\{T_{n}\}_{n\geq 0}$. Note that the expectation $\mathbf d_{\scr M}$ of an increment $Y_k:=\tau_k^\ttu-\tau_k^\ttd$ of $M$ is meaningful whenever (at least) one of the persistence times is integrable. By contrast,  the mean $\mathbf d_{\scr T}$ of a jump $\tau_{k}:=\tau_{k}^{\ttu}+\tau_{k}^{\ttd}$ related to $T$ is always defined.
Finally, we set when it makes sense, 
\begin{equation}
\label{drift-def2}
{\mathbf d}_{\scr S}:=
\frac{\mathbf d_{\scr M}}{\mathbf d_{\scr T}}
=\frac{\E[\tau^{\ttu}_{1}]-\E[\tau^{\ttd}_{1}]}{\E[\tau^{\ttu}_{1}]+\E[\tau^{\ttd}_{1}]},
\end{equation}
extended by continuity to $\pm 1$ if only one of the persistence times has a finite expectation.  This characteristic naturally arises in the recurrence and transience features and as an almost sure drift of $S$ as it is shown in \cite{PRWI}.

%%%%%%%%%%%%%%%%%%%%%%%%%%%%%%%%%%%%%%%%%%%%%%%%%%%%%%%%%%%%%%%%%%%%%%%%%%%%%%%%%%%%%%%%%%%%%%%%%%%%%%%%%%%%%%%%%%%%%%%%%%%%%%%%%%%%%%%%%%%%%%%%%%%%%%%%%%%%%%%%%%%%%%%%%%%%%%%%%%%%%
%%%%%%%%%%%%%%%%%%%%%%%%%%%%%%%%%%%%%%%%%%%%%%%%%%%%%%%%%%%%%%%%%%%%%%%%%%%%%%%%%%%%%%%%%%%%%%%%%%%%%%%%%%%%%%%%%%%%%%%%%%%%%%%%%%%%%%%%%%%%%%%%%%%%%%%%%%%%%%%%%%%%%%%%%%%%%%%%%%%%%%%%%%%%%%%%%%%%%%%%%%%%%%%%%%%%%%%%%%%%%%%%%%%%%%%%%%%%%%%%%%%%%%%%%%%%%%%%%%%%%%%%%%%%%%%%%%%%%%%%%%%%%%%%%%%%%%%%%%%%%%%%%%%%%%%%%%%%%%%%%%%%%%%%%%%%%%%%%%%

\subsection{Mean drift  and stability assumption}

\label{Meandrift}

To go further and introduce the suitable centering term  $\mathbf m_{\scr S}$ in the scaling limits (\ref{scalingforme}) we need to consider, when it exists, the tail balance parameter defined by
\begin{equation}\label{balance0}
{\mathbf b}_{\scr S}:=\lim_{t\to\infty}\frac{\tail_{\ttu}(t)-\tail_{\ttd}(t)}{\tail_{\ttu}(t)+\tail_{\ttd}(t)},
\end{equation} 
and set 
\begin{equation}\label{meandriftass}
{\mathbf m}_{\scr S}:=\left\{\begin{array}{ll}
\mathbf b_{\scr S},  & \mbox{when $\tau^{\ttu}_1$ and $\tau^{\ttd}_1$ are both not integrable,}\\
\mathbf d_{\scr S},  & \mbox{otherwise.}  
\end{array}\right.	
\end{equation}	
In the light of the $\rm L^1$-convergence in (\ref{meandrift1}) below, this term is naturally called the mean drift of $S$. Note also that it generalizes $\mathbf d_{\scr S}$ since, when it is well defined, 
\begin{equation}\label{meandriftass2}
{\mathbf m}_{\scr S}=\lim_{t\to\infty}\frac{\Theta_{\ttu}(t)-\Theta_{\ttd}(t)}{\Theta_{\ttu}(t)+\Theta_{\ttd}(t)}.
\end{equation}	

A Strong Law of Large Number (SLLN) is established in \cite{peggy} as well as a (non-functional) CLT under some strong moment conditions on the running times $\tau^{\ell}_{1}$. Here the assumptions are drastically weakened. More precisely, our main hypothesis to get functional invariance principles -- assumed throughout the article unless otherwise stated -- is the following

\begin{ass}[$\alpha$-stability]\label{assglobal} The mean drift $\mathbf m_{\scr S}$ is well defined and  not extremal, that is $\mathbf m_{\scr S}\in(-1,1)$. Moreover, there exists $\alpha\in(0,2]$ such that
\begin{equation}\label{cond10}
\tau^{\mathtt c}_{1}:=(1-\mathbf m_{\scr S})\tau^{\ttu}_{1}-(1+\mathbf m_{\scr S})\tau^{\ttd}_{1}\in{\rm D}(\alpha),	
\end{equation}
{\it i.e.\@} $\tau^{\mathtt c}_1$ belongs to the domain of attraction of an $\alpha$-stable distribution.
\end{ass}

Note that $\tau^{\mathtt c}_{1}$ is a centered random variable when $\mathbf m_{\scr S}=\mathbf d_{\scr S}$. When $\mathbf m_{\scr S}=\mathbf b_{\scr S}$, we will see that it is always -- in some sense -- well  balanced. 
Obviously, the $\alpha$-stable distribution in the latter hypothesis is supposed to be non-degenerate. Closely related to stable distributions and their domains of attraction are the notions of regularly varying functions, infinitely divisible distributions and L\'{e}vy processes. We refer, for instance, to \cite{BGT,Sato,Applebaum,Skorohod,Feller,WW} for a general panorama.

%%%%%%%%%%%%%%%%%%%%%%%%%%%%%%%%%%%%%%%%%%%%%%%%%%%%%%%%%%%%%%%%%%%%%%%%%%%%%%%%%%%%%%%%%%%%%%%%%%%%%%%%%%%%%%%%%%%%%%%%%%%%%%%%%%%%%%%%%%%%%%%%%%%%%%%%%%%%%%%%%%%%%%%%%%%%%%%%%%%%%
%%%%%%%%%%%%%%%%%%%%%%%%%%%%%%%%%%%%%%%%%%%%%%%%%%%%%%%%%%%%%%%%%%%%%%%%%%%%%%%%%%%%%%%%%%%%%%%%%%%%%%%%%%%%%%%%%%%%%%%%%%%%%%%%%%%%%%%%%%%%%%%%%%%%%%%%%%%%%%%%%%%%%%%%%%%%%%%%%%%%%%%%%%%%%%%%%%%%%%%%%%%%%%%%%%%%%%%%%%%%%%%%%%%%%%%%%%%%%%%%%%%%%%%%%%%%%%%%%%%%%%%%%%%%%%%%%%%%%%%%%%%%%%%%%%%%%%%%%%%%%%%%%%%%%%%%%%%%%%%%%%%%%%%%%%%%%%%%%%%

\subsection{Normalizing functions}

\label{normalizing}

For the statement of the functional invariance principle, it is necessary to define the appropriate sub-linear normalizing function $\lambda(u)$ in the functional convergences  (\ref{scalingforme}). To this end, we introduce the non-negative and non-decreasing functions $\Sigma(t)$ and $\Theta(t)$ respectively given by
\begin{equation}\label{sigma0} 
\Sigma(t)^{2}:=(1-\mathbf m_{\scr S}){V}_{\ttu}\left(\frac{t}{1-\mathbf m_{\scr S}}\right)+ 
(1+\mathbf m_{\scr S}){V}_{\ttd}\left(\frac{t}{1+\mathbf m_{\scr S}}\right)
\mbox{and}\quad \Theta(t):=\Theta_{\ttu}(t)+\Theta_{\ttd}(t).
\end{equation}
Then we shall see that a relevant choice can be  given  by setting $\lambda(u):=a\circ s(u)$ with
\begin{equation}\label{normalizing0}
a(u):=\inf\left\{ t > 0 : \frac{t^{2}}{\Sigma(t)^{2}}\geq u\right\}\quad\mbox{and}\quad s(u):=\inf \left\{t>0 : \Theta\circ a(t)\,t\geq u \right\}.	
\end{equation}

\section{Statements of the results}

\setcounter{equation}{0}

\label{sectionresults}

We establish -- under the $\alpha$-stability Assumption \ref{assglobal} -- a general functional invariance principle, stated in its compact form in Theorem \ref{thmgene0} below. In this one, the Skorohod space of all right-continuous functions $\omega : [0,\infty)\longrightarrow \mathbb R$ having left limits ({\it c\`{a}dl\`{a}g}) is endowed with the ${\rm M}_{1}$-topology (making it a Polish space) and the convergence in distribution with respect to the related Borel $\sigma$-fields is denoted by ${\rm M_{1}}$. For more details, we refer especially to  \cite{WW}.  

\begin{theo}\label{thmgene0} There exists a non-trivial {\it c\`{a}dl\`{a}g\@} process $Z_{\alpha}$ such that 
\begin{equation}\label{scalingform}
\left\{\frac{S_{\lfloor u\,t\rfloor}-\mathbf m_{\scr S}\,u\,t}{\lambda(u)}\right\}_{t\geq 0}\quad\mbox{and}\quad \left\{\frac{S_{u\,t}-\mathbf m_{\scr S}\,u\,t}{\lambda(u)}\right\}_{t\geq 0}\;\xRightarrow[u\to\infty]{{\rm M}_{1}}\; \{Z_{\alpha}(t)\}_{t\geq 0}.
\end{equation}
\end{theo}

\begin{rem}
We point out that we make the effort to give an unified functional convergence. To enforce the scaling limit, we do not have to know the index of stability $\alpha$ since computing the mean drift $\mathbf m_{\scr S}$ and the normalization function $\lambda(u)$ only involve to observe the truncated mean and second moment of the running times. This could have some statistical interests.
\end{rem}

Furthermore,  it turns out that there are mainly four situations  according to the position of the index of stability $\alpha\in(0,2]$ with respect to  the partition
$(0,1)\sqcup\{1\} \sqcup (1,2)\sqcup \{2\}$. More precisely, reading the four latter intervals from the right to the left,  the limit stochastic process $Z_{\alpha}$ can be equal to
\begin{enumerate}
\item[1)] $B$ -- a brownian motion;
\item[2)] $S_{\alpha,\beta}$ -- a strictly $\alpha$-stable L\'{e}vy process with skewness parameter $\beta$;
\item[3)] $\mathfrak C$ -- a symmetric Cauchy process;
\item[4)] ${\mathcal S}_{\alpha}$ -- an arcsine Lamperti anomalous diffusion. 
\end{enumerate}

The first three situations cover all classical strictly stable L\'{e}vy processes having infinite variations and we will specify later what is meant by anomalous diffusion and the arcsine Lamperti terminology, which is undoubtedly the most fruitful situation. Note that the authors in \cite{peggy} show the convergence of some rescaling  toward a kind of non-symmetric generalized telegraph process.  In this paper the limit processes are no longer generalizations of the telegraph process since the probabilities of changing directions are fixed. Here actual scaling limits are investigated. Theorem \ref{thmgene0} is divided and completed below according to the four latter possible situations. In particular, we consider stronger convergence in distribution: the ${\rm J}_{1}$-convergence and  its restriction $\mathcal C$ to the Wiener space of all continuous functions (see \cite{Billcvg} for instance).

%%%%%%%%%%%%%%%%%%%%%%%%%%%%%%%%%%%%%%%%%%%%%%%%%%%%%%%%%%%%%%%%%%%%%%%%%%%
%%%%%%%%%%%%%%%%%%%%%%%%%%%%%%%%%%%%%%%%%%%%%%%%%%%%%%%%%%%%%%%%%%%%%%%%%%%%%%%%%%%%%%%%%%%%%%%%%%%%%%%%%%%%%%%%%%%%%%%%%%%%%%%%%%%%%%%%%%%%%%%%%%%%%%%%%%%%%%%%%%%%%%%%%%%%%%%%%%%%%%%%%%%%%%%%%%%%%%%%%%%%%%%%%%%%%%%%%%%%%%%%%%%%%%%%%%%%%%%%%%%%%%%%%%%%%%%%%%%%%%%%%%%%%%%%%%%%%%%%%%%%%%%%%%%%%%%%%%%%%%%%%%%%%%%%%%%%%%%%%%%%%%%%%%%%%%%%%%%
%%%%%%%%%%%%%%%%%%%%%%%%%%%%%%%%%%%%%%%%%%%%%%%%%%%%%%%%%%%%%%%%%%%%%%%%%%%
%%%%%%%%%%%%%%%%%%%%%%%%%%%%%%%%%%%%%%%%%%%%%%%%%%%%%%%%%%%%%%%%%%%%%%%%%%%%%%%%%%%%%%%%%%%%%%%%%%%%%%%%%%%%%%%%%%%%%%%%%%%%%%%%%%%%%%%%%%%%%%%%%%%%%%%%%%%%%%%%%%%%%%%%%%%%%%%%%%%%%%%%%%%%%%%%%%%%%%%%%%%%%%%%%%%%%%%%%%%%%%%%%%%%%%%%%%%%%%%%%%%%%%%%%%%%%%%%%%%%%%%%%%%%%%%%%%%%%%%%%%%%%%%%%%%%%%%%%%%%%%%%%%%%%%%%%%%%%%%%%%%%%%%%%%%%%%%%%%%

\subsection{Classical L\'{e}vy situation}

\label{brownian}

Here we deepen the case when the limit process is a strictly stable L\'{e}vy motion of infinite variation. We shall prove -- in the generic situation $\alpha\in(1,2)$ -- that up to some scale parameter it is an $\alpha$-stable L\'{e}vy process $S_{\alpha,\beta}$ with  skewness parameter
\begin{equation}\label{skewnessexpress}
\beta=\frac{(1-\mathbf m_{\scr S})^{\alpha}(1+\mathbf b_{\scr S})-(1+\mathbf m_{\scr S})^{\alpha}(1-\mathbf b_{\scr S})}{(1-\mathbf m_{\scr S})^{\alpha}(1+\mathbf b_{\scr S})+(1+\mathbf m_{\scr S})^{\alpha}(1-\mathbf b_{\scr S})}.
\end{equation}
More precisely, we shall see that its symbol $\log\mathbb E[e^{i u S_{\alpha,\beta}(1)}]$ is equal to
\begin{equation}\label{symbollevy0}
-\frac{(2-\alpha)\Gamma(2-\alpha)}{\alpha}\frac{\sin\left(\frac{\pi}{2}(\alpha-1)\right)}{\alpha-1} \,\bigg[1-i\beta\tan\left(\frac{\pi\alpha}{2}\right)\bigg] \,|u|^{\alpha},
\end{equation}
where $\Gamma$ is the Gamma function. In particular, its L\'{e}vy jump measure  is given by  
\begin{equation}\label{jumpmeasure00}
\left[\left(\frac{1-\beta}{2}\right)\mathds 1_{\{x<0\}}+\left(\frac{1+\beta}{2}\right)\mathds 1_{\{x>0\}}\right]  \frac{(2-\alpha)}{|x|^{\alpha+1}}\,dx.
\end{equation}

Regarding the  case $\alpha=2$, the log-characteristic function of the limit process is still given by the right hand side of (\ref{symbollevy0}) -- whatever the value of $\beta$ is -- and it is nothing but a standard Brownian motion denoted by $B$. Concerning the last situation $\alpha=1$, the symbol is given by the right hand side of (\ref{symbollevy0}) with a null skewness parameter $\beta=0$. In other words, 
it is a symmetric Cauchy  process denoted by $\mathfrak C$ and characterized by its marginal distribution 
\begin{equation}
\mathfrak C(1)\sim \frac{1}{2}\frac{1}{(\pi/2)^{2}+x^{2}}\,dx.	
\end{equation}

As a matter of fact, the skewness parameter (\ref{skewnessexpress}) is  associated with the random variable $\tau_1^{\mathtt c}$ introduced in (\ref{cond10}). This expression will be explained in the last section  where equivalent formulations of Assumption \ref{assglobal} can be deduced from Lemma \ref{lemmeequi}. Besides, it also comes from this lemma -- together with Lemma \ref{lemmebound} and classical results on regular variations in \cite{BGT} for instance -- that 
$\mathbf d_{\scr T}<\infty$ and $\mathbf m_{\scr S}=\mathbf d_{\scr S}$ when  $\alpha\in(1,2)$ but also that there exists   a slowly varying function $\Xi_{\alpha}(u)$ such that   
\begin{equation}
a(u)\;\underset{u\to\infty}{\sim}\;\Xi_{\alpha}(u)\, u^{1/\alpha}\quad\mbox{and}\quad s(u)\;\underset{u\to\infty}{\sim}\;\frac{u}{\mathbf d_{\scr T}},
\end{equation}
In particular, the random variable $\tau_1^{\mathtt c}$ is centered. As far as $\alpha=2$, the latter considerations also hold but with an ultimately non-decreasing slowly varying function $\Xi_2(u)$ since it can be chosen as $\Sigma\circ a(u)$. 

When $\alpha=1$, it is possible for both of the running times $\tau_1^\ttu$ and $\tau_1^\ttd$ to be integrable or not. In the first situation, the latter estimations are still effective. In the second one, the random variable $\tau_{1}^{\mathtt c}$ is no longer centered but it is well balanced -- the skewness parameter $\beta$ is equal to zero -- in the sight of equality (\ref{skewnessexpress}) since $\mathbf m_{\scr S}=\mathbf b_{\scr S}$. Anyway, when $\alpha=1$, we can formulate  the following

\begin{lem}\label{cauchylemme} When $\alpha=1$ there exist two  slowly varying  functions $\Xi_{1}(u)$ and $D(u)$, the latter being  ultimately non-decreasing, such that
\begin{equation}\label{Du}
a(u)\;\underset{u\to\infty}{\sim}\; \Xi_{1}(u)\, u\quad\mbox{and}\quad s(u)\;\underset{u\to\infty}{\sim}\;\frac{u}{D(u)},\quad\mbox{with}\quad \lim_{u\to\infty}\frac{D(u)}{\Xi_{1}\circ s(u)}=\infty.
\end{equation}
\end{lem}

The proof of these estimates -- especially the right hand side of (\ref{Du}) which is not a direct consequence of regular variations -- is postponed to the end of Section \ref{sectionl\'{e}vy}. 

\begin{theo}[classical L\'{e}vy situation] 
\label{scalingbrownian} Suppose $\alpha\in[1,2]$, then the scaling limits in (\ref{scalingform}) can be rewritten -- in a stronger way -- as follows.
\begin{enumerate}
\item[1)] If $\alpha=2$ -- the Gaussian case -- then    
\begin{multline}\label{scalingBM2000}
\left\{{\frac{\sqrt{{\mathbf d}_{\scr T}\,u}}{\Xi_{2}(u)}}
\left(\frac{S_{\lfloor u\,t\rfloor}}{u}-{{\mathbf d}_{\scr S}}\, t\right)\right\}_{t\geq 0} 
\;\xRightarrow[u\to\infty]{\rm J_{1}}\;
\{B(t)\}_{t\geq 0},\\
\mbox{and}\quad \left\{{\frac{\sqrt{{\mathbf d}_{\scr T}}\sqrt{u}}{\Xi_{2}(u)}}
\left(\frac{S_{u\,t}}{u}-{{\mathbf d}_{\scr S}}\, t\right)\right\}_{t\geq 0} 
\;\xRightarrow[u\to\infty]{\mathcal C}\;
\{B(t)\}_{t\geq 0}.	
\end{multline}
\item[2)] If $\alpha\in(1,2)$ -- the generic case -- then
\begin{equation}\label{scalinggeneric}
\left\{
\frac{\mathbf d_{\scr T}^{{1}/{\alpha}}\,u^{1-{1}/{\alpha}}}{\Xi_{\alpha}(u)}
\left(\frac{S_{\lfloor u \,t\rfloor}}{u}-{\mathbf d}_{\scr S}\,t\right)\right\}_{t\geq 0} 
\;\xRightarrow[u\to\infty]{\rm J_{1}}\; 
\{ S_{\alpha,\beta}(t) \}_{t\geq 0}.
\end{equation}
\item[3)] If $\alpha=1$ -- the Cauchy case -- then
\begin{equation}\label{scalingcauchy2}
\left\{
\frac{D(u)}{\Xi_{1}\circ s(u)}
\left(\frac{S_{\lfloor u \,t\rfloor}}{u}-{\mathbf m}_{\scr S}\,t\right)\right\}_{t\geq 0} 
\;\xRightarrow[u\to\infty]{\rm J_{1}}\; 
\{ 
\mathfrak C(t) \}_{t\geq 0}.
\end{equation}
\end{enumerate}
\end{theo}

%%%%%%%%%%%%%%%%%%%%%%%%%%%%%%%%%%%%%%%%%%%%%%%%%%%%%%%%%%%%%%%%%%%%%%%%%%%
%%%%%%%%%%%%%%%%%%%%%%%%%%%%%%%%%%%%%%%%%%%%%%%%%%%%%%%%%%%%%%%%%%%%%%%%%%%%%%%%%%%%%%%%%%%%%%%%%%%%%%%%%%%%%%%%%%%%%%%%%%%%%%%%%%%%%%%%%%%%%%%%%%%%%%%%%%%%%%%%%%%%%%%%%%%%%%%%%%%%%%%%%%%%%%%%%%%%%%%%%%%%%%%%%%%%%%%%%%%%%%%%%%%%%%%%%%%%%%%%%%%%%%%%%%%%%%%%%%%%%%%%%%%%%%%%%%%%%%%%%%%%%%%%%%%%%%%%%%%%%%%%%%%%%%%%%%%%%%%%%%%%%%%%%%%%%%%%%%%

%%%%ICI 28/09/16

\subsection{Anomalous situation}

In this section, we detail the functional convergence (\ref{scalingform}) when  $\alpha\in(0,1)$. It turns out  that the limit process is no longer a stable L\'{e}vy process nor even a Markov process. Roughly speaking, it is -- up to a linear and multiplicative term  -- a random continuous piecewise linear function built from an $\alpha$-stable subordinator as follows.

\begin{enumerate}
\item[1)] We attach {\it i.i.d.\@} Bernoulli distributions of parameter $(1+\mathbf b_{\scr S})/2$ to every maximal open intervals  in the complementary the range of the $\alpha$-stable subordinator. 
\item[2)] According to their success or their failure, the slope on the corresponding jump interval is chosen to be equal to $1$ or $-1$ accordingly. 
\end{enumerate}

We refer carefully to Figure \ref{anomal} illustrating some notations and the construction of this process. Again, classical results on regularly varying functions together with Lemma \ref{lemmeequi} and \ref{lemmebound} apply and we can check that  $\mathbf d_{\scr T}=\infty$ and $\mathbf m_{\scr S}=\mathbf b_{\scr S}$ in this section. The reason  we used $\mathbf m_{\scr S}$ rather than $\mathbf b_{\scr S}$ or vice versa is motivated by the will to focus on different meanings -- a mean drift or a balance term  -- according to the situations. In the following, we mention \cite{Bertoin} concerning classical results about L\'{e}vy subordinators and their local times and  \cite[Chap. 6.]{ItoMcKean} for general Markov processes. %\cite[Chap. 8.]{Sharpe} and \cite{FBT}. 
Finally, we allude to \cite{ItoPointMarkov} for the general theory of excursions and Poisson point processes which is strongly connected to these notions.

\begin{figure}[!b]
\centering
\includegraphics[width=124mm]{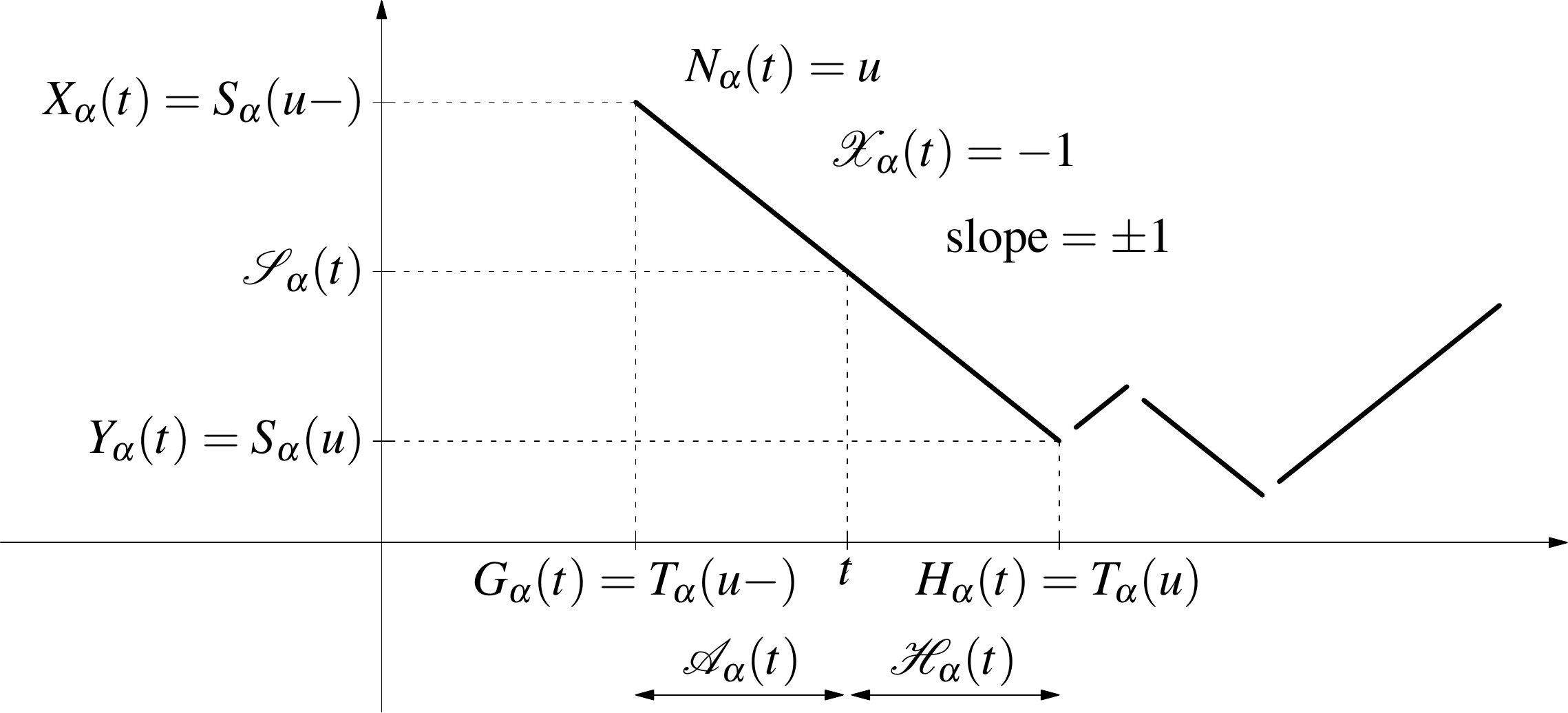}
\caption{\label{anomal}Construction of the anomalous diffusion}
\end{figure}

%%%%%%%%%%%%%%%%%%%%%%%%%%%%%%%%%%%%%%%%%%%%%%%%%%%%%%%%%%%%%%%%%%%%%%%%%%%
%%%%%%%%%%%%%%%%%%%%%%%%%%%%%%%%%%%%%%%%%%%%%%%%%%%%%%%%%%%%%%%%%%%%%%%%%%%%%%%%%%%%%%%%%%%%%%%%%%%%%%%%%%%%%%%%%%%%%%%%%%%%%%%%%%%%%%%%%%%%%%%%%%%%%%%%%%%%%%%%%%%%%%%%%%%%%%%%%%%%%%%%%%%%%%%%%%%%%%%%%%%%%%%%%%%%%%%%%%%%%%%%%%%%%%%%%%%%%%%%%%%%%%%%%%%%%%%%%%%%%%%%%%%%%%%%%%%%%%%%%%%%%%%%%%%%%%%%%%%%%%%%%%%%%%%%%%%%%%%%%%%%%%%%%%%%%%%%%%%

\subsubsection{Preliminaries}

\label{subordinator}

Let $T_{\alpha}$ be an $\alpha$-stable L\'{e}vy subordinator with no drift and recall that it is a {\it c\`{a}dl\`{a}g\@} non-decreasing pure jump process. Let $\mathcal J$ be its  random set of jumps -- the random set of discontinuity points. It follows that the closure  of the  image   $\mathcal R_{\alpha}:=\{T_{\alpha}(t) : t\geq 0\}$ -- the so-called regenerative set or range -- is a perfect set of zero Lebesgue measure satisfying
\begin{equation}
\overline{\mathcal R_{\alpha}}=\mathcal R_{\alpha}\sqcup\left\{T_{\alpha}(u-): u\in\mathcal J\right\}\quad\mbox{and}\quad [0,\infty)\setminus \overline{\mathcal R_{\alpha}}=\bigsqcup_{u\in \mathcal J}(T_{\alpha}(u-),T_{\alpha}(u)).
\end{equation}

The backward recurrence time (the current life time) and the forward renewal time (the residual life time) processes are respectively defined as
\begin{equation}
{G}_{\alpha}(t):=\sup\{s\leq t : s\in\mathcal R_{\alpha}\}\quad\mbox{and}\quad H_{\alpha}(t)=\inf\{s\geq t : s\in\mathcal R_{\alpha}\}.
\end{equation}
Those are {\it c\`{a}dl\`{a}g\@} stochastic processes and we can check that  
%$G_{\alpha}(t)\leq t<H_{\alpha}(t)$, when  $t\notin \mathcal R_{\alpha}$,  with equalities  on the range. More precisely,  
$G_{\alpha}(t)=T_{\alpha}(u-)$ and $H_{\alpha}(t)=T_{\alpha}(u)$ when $ T_{\alpha}(u-)\leq t< T_{\alpha}(u)$ whereas $G_{\alpha}(t)=H_{\alpha}(t)=T_{\alpha}(u)$ when $t=T_{\alpha}(u)$. To go further, we consider the local time of (the inverse of) the stable subordinator $T_{\alpha}$  defined  by
\begin{equation}\label{localtime}
N_{\alpha}(t):=\inf\{s>0 : T_{\alpha}(s)>t\}=\sup\{s>0 : T_{\alpha}(s)\leq t\}.
\end{equation}
Note the usefull switching identity
$\{T_{\alpha}(s)\leq t\}=\{s\leq N_{\alpha}(t)\}$ and the fact that $N_\alpha$  is a continuous stochastic process. When $\alpha=1/2$ it is nothing but the classical local time of a Brownian motion. In full generality,  it is is a non-decreasing  continuous stochastic process such that the support of the random Stieljes measure $dN_{\alpha}$ is equal to the  range of the subordinator. Then, one has $N_{\alpha}(t)=u$ when $ T_{\alpha}(u-)\leq t\leq T_{\alpha}(u)$ and the first and last renewal time processes can be rewritten as
\begin{equation}\label{anomalousdiff0}
{G}_{\alpha}(t)= 
[T_{\alpha}^{-}\circ N_{\alpha}]^{+}(t)
\quad\mbox{and}\quad  H_{\alpha}(t)=T_{\alpha}\circ N_{\alpha}(t),
\end{equation}
where $F^{\pm}(t):=F(t\pm)$ is the right or the left continuous version of a function $F(t)$. 

Furthermore, it is well-known that stable subordinator can be decomposed as
\begin{equation}\label{decomposub}
T_{\alpha}(t)=T_{\alpha}^{\ttu}(t)+T_{\alpha}^{\ttd}(t):= \left(\frac{1+\mathbf b_{\scr S}}{2}\right)^{1/\alpha} T_{\alpha}^{\prime}(t)+\left(\frac{1-\mathbf b_{\scr S}}{2}\right)^{1/\alpha} T_{\alpha}^{\prime\prime}(t),	
\end{equation}
where $T_{\alpha}^{\prime}$  and $T_{\alpha}^{\prime\prime}$ are {\it i.i.d.\@} with the same distribution as $T_{\alpha}$. This decomposition can be obtained intrinsically from $T_{\alpha}$ by labelling each  jump interval (an excursion) $I=(T_{\alpha}(u-),T_{\alpha}(u))$ as in \cite[pp. 342-343]{PitYor}  by {\it i.i.d.\@} Rademacher random variables $\mathcal X_{I}$ independent of $T_{\alpha}$ and of parameter $(1+\mathbf b_{\scr S})/2$. It follows that $T_{\alpha}^{\ttu}(t)$ and $T_{\alpha}^{\ttd}(t)$ can be viewed as the sums up to the time $t$ of the jumps $\Delta T_\alpha(u):=T_\alpha(u)-T_\alpha(u-)$ for which the corresponding labels are respectively equal to one and minus one. They are both thinning of the initial subordinator.  Coupled with the latter decomposition, we can consider the $\alpha$-stable L\'{e}vy process
\begin{equation}
S_{\alpha}(t):= T_{\alpha}^{\ttu}(t)-T_{\alpha}^{\ttd}(t)=\left(\frac{1+\mathbf b_{\scr S}}{2}\right)^{1/\alpha} T_{\alpha}^{\prime}(t)-\left(\frac{1-\mathbf b_{\scr S}}{2}\right)^{1/\alpha} T_{\alpha}^{\prime\prime}(t).	
\end{equation}
This one can be viewed as a pure jump process whose increment $\Delta S_{\alpha}(u)$ is equal to $\Delta T_{\alpha}(u)$ or $-\Delta T_{\alpha}(u)$  according to the label of the corresponding excursion. Following the terminology used in \cite{Straka} we can introduce the so-called lagging and leading {\it c\`{a}dl\`{a}g\@} stochastic processes respectively defined by 
\begin{equation}\label{anomalousdiff2}
{X}_{\alpha}(t):= 
[S_{\alpha}^{-}\circ N_{\alpha}]^{+}(t)
\quad\mbox{and}\quad  Y_{\alpha}(t):=S_{\alpha}\circ N_{\alpha}(t).
\end{equation}
It is not difficult to check  that $X_{\alpha}(t)=S_{\alpha}(u-)$ and $Y_{\alpha}(t)=S_{\alpha}(u)$ when $T_{\alpha}(u-)\leq t<T_{\alpha}(u)$ whereas $X_{\alpha}(t)=Y_{\alpha}(t)=S_{\alpha}(u)$ when $t=T_{\alpha}(u)$.

%%%%%%%%%%%%%%%%%%%%%%%%%%%%%%%%%%%%%%%%%%%%%%%%%%%%%%%%%%%%%%%%%%%%%%%%%%%
%%%%%%%%%%%%%%%%%%%%%%%%%%%%%%%%%%%%%%%%%%%%%%%%%%%%%%%%%%%%%%%%%%%%%%%%%%%%%%%%%%%%%%%%%%%%%%%%%%%%%%%%%%%%%%%%%%%%%%%%%%%%%%%%%%%%%%%%%%%%%%%%%%%%%%%%%%%%%%%%%%%%%%%%%%%%%%%%%%%%%%%%%%%%%%%%%%%%%%%%%%%%%%%%%%%%%%%%%%%%%%%%%%%%%%%%%%%%%%%%%%%%%%%%%%%%%%%%%%%%%%%%%%%%%%%%%%%%%%%%%%%%%%%%%%%%%%%%%%%%%%%%%%%%%%%%%%%%%%%%%%%%%%%%%%%%%%%%%%%

\subsubsection{Arcsine Lamperti anomalous diffusions}

We can now define the so-called Arcsine Lamperti anomalous diffusion by
\begin{equation}\label{anomalousdiff}
\mathcal S_{\alpha}(t):=X_{\alpha}(t)+\frac{t-G_{\alpha}(t)}{H_{\alpha}(t)-G_{\alpha}(t)}(Y_{\alpha}(t)-X_{\alpha}(t)),	
\end{equation}
with $\mathcal S_{\alpha}(t)=X_{\alpha}(t)$ when $t\in\mathcal R_{\alpha}$. In other words, the so-called anomalous diffusion $\mathcal S_{\alpha}(t)$ is the center of mass of the lagging and leading processes $X_{\alpha}(t)$ and $Y_{\alpha}(t)$ with respective weights given by the so-called remaining time and age time  processes 
\begin{equation}
\mathcal H_{\alpha}(t):=H_{\alpha}(t)-t\quad\mbox{and}\quad \mathcal A_{\alpha}(t)=t-G_{\alpha}(t).	
\end{equation}

This kind of construction as already been exposed in \cite{Magdziarz}. From the self-similarity of $T_\alpha$ we get that the law of $\mathcal S_{\alpha}$ is self-similar of index $1$: for any $\lambda>0$, 
\begin{equation}
\{\mathcal S_\alpha(\lambda t)\}_{t\geq 0}\;\overset{\mathcal L}{=}\;\{\lambda \mathcal S_\alpha(t)\}_{t\geq 0}.
\end{equation}

Besides, its law does not depend on the scale parameter chosen to define $T_\alpha$. To go further, it is continuous and of bounded variation and it can be written as
\begin{equation}\label{limitintegral}
\mathcal S_{\alpha}(t)=\int_{0}^{t}\mathcal X_{\alpha}(s)\,ds,
\end{equation}
where $\{\mathcal X_{\alpha}(t)\}_{t\geq 0}$ denote the so-called label process taking its values in $\{-1,1\}$ and such that $\mathcal X_{\alpha}(t):=\mathcal X_{I}$ when $t\in I$ -- no matter its values on  the regenerative set is. We recall that the $\mathcal X_I$ are {\it i.i.d.\@} Rademacher random variables of parameter $(1+\mathbf b_{\scr S})/2$ attached on the excursion intervals $I$. To get the latter integral representation one can see that  $\mathcal S_{\alpha}$ is linear on each excursion interval $I$ with a slope given by $\mathcal X_I$. Then, noting that  there is clearly equality on the random set $\{H_\alpha(t) : t>0\}$ with $Y_\alpha$
 -- the so-called overshoot continuous time random walk -- we obtain the equality on $[0,\infty)$. 
 
 Furthermore, one can see that $\lambda (u)=a\circ s(u)$ is equivalent to $(2-\alpha)(1-\alpha)u/\alpha$ as $u$ goes to infinity and thus  we only need to study the asymptotic of normalized processes $t\mapsto S_{u\,t}/u$. 

%\begin{equation}
%c_{\alpha}:= \frac{(2-\alpha)(1-\alpha)}{\alpha}\left[(1-\mathbf m_{\scr S})^{\alpha}\left(\frac{1+\mathbf b_{\scr S}}{2}\right)+(1+\mathbf m_{\scr S})^{\alpha}\left(\frac{1-\mathbf b_{\scr S}}{2}\right)\right]^{1/\alpha}.	
%\end{equation}

%~ \begin{prop}
%~ There exits a unique invariant measure $\mathbb Q_{\alpha}$ for the regenerative set $\overline{\mathcal R_{\alpha}}$
%~ \end{prop}

\begin{theo}[anomalous situation]\label{complement}
Suppose $\alpha\in(0,1)$, then the functional convergence (\ref{scalingform}) can be rewritten -- in a stronger form -- as
\begin{equation}\label{arcsinescaling}
\left\{
\frac{S_{u \,t}}{u}\right\}_{t\geq 0} 
\;\xRightarrow[u\to\infty]{\mathcal C}\; 
\left\{{\mathcal S}_{\alpha}(t)\right\}_{t\geq 0}.
\end{equation}	
Moreover, the marginal $\mathcal S_{\alpha}(t)$ has a density function $f_t(x)$ on $(-t,t)$ satisfying -- in a weak sense -- the fractional partial differential equation   
\begin{multline}\label{goveq}
\left[\frac{1+{\mathbf m_{\scr S}}}{2}\left(\frac{\partial }{\partial x}+\frac{\partial }{\partial t}\right)^\alpha + \frac{1-{\mathbf m_{\scr S}}}{2}\left(\frac{\partial }{\partial x}-\frac{\partial }{\partial t}\right)^\alpha\right] f_t(x)\\=\frac{1}{\Gamma(1-\alpha)t^\alpha}\left[\frac{1+\mathbf m_{\scr S}}{2}\delta_{\,t}(dx)+\frac{1-\mathbf m_{\scr S}}{2}\delta_{-t}(dx)\right],
\end{multline}
In addition, the latter density admits the integral representation   
\begin{equation}\label{potential}
f_t(x)=\bigintsss_0^t \left[{\frac{1+\mathbf m_{\scr S}}{2}}\frac{u_\alpha(x-(t-s),s)}{\Gamma(1-\alpha)}+\frac{1-\mathbf m_{\scr S}}{2}\frac{u_\alpha(x+(t-s),s)}{\Gamma(1-\alpha)}\right]\frac{ds}{(1-s)^\alpha},
\end{equation}
where $u_\alpha(x,t)$ denotes the $0$-potential density of $(S_\alpha,T_\alpha)$. Furthermore, this it is -- up to an affine transformation -- an arsine Lamperti distribution given by
\begin{equation}\label{density}
f_t(x)=\frac{2\sin(\pi \alpha)}{\pi\,t}\frac{(t-x)^{\alpha-1}(t+x)^{\alpha-1}}{\mathbf r_{\scr S}(t-x)^{2\alpha}+2\cos(\pi \alpha)(t+x)^{\alpha}(t-x)^{\alpha}+\mathbf r_{\scr S}^{-1}\,(t+x)^{2\alpha}},
\end{equation}
with  $\mathbf r_{\scr S}:=(1+\mathbf m_{\scr S})/(1-\mathbf m_{\scr S})$. Also, the stochastic process $(\mathcal X_{\alpha},\mathcal A_{\alpha})$ on the state space $\{-1,1\}\times [0,\infty)$ is Markovian 
\end{theo}

\subsubsection{On general PRWs}

In \cite[Section 4.]{PRWI} the recurrence and transience criteria are extended to a wider class of PRWs -- without any renewal assumption -- suitable  perturbations  of the double-infinite comb model when the persistence times are non-integrable. In the same way, it would be possible -- adding some other natural stability assumptions  -- to deduce the same theorem for these walks as it is stated below.

\begin{figure*}[!b]
\centering
\includegraphics[width=124mm]{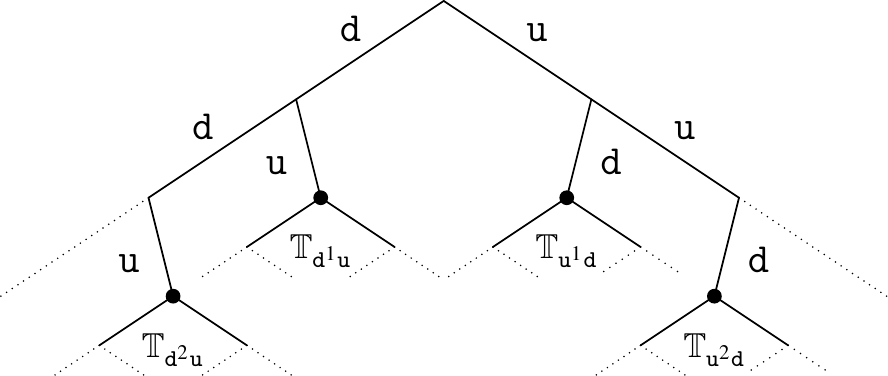}
\caption{Grafting of the double-infinite comb}
\label{fig:3}
\end{figure*}

Consider a double-infinite comb and attach to each finite leaf $c$ another context tree $\mathbb T_{c}$ (possibly trivial) as in Figure \ref{fig:3}. The leaves of the related graft are denoted by $\mathcal C_{c}$ and this one is endowed with Bernoulli distributions $\{q_{l} : l\in\mathcal C_{c}\}$ on $\{\ttu,\ttd\}$. Note that any probabilized context tree on $\{\ttu,\ttd\}$ can be constructed in this way. We denote by $S^{\mathtt g}$ the corresponding PRW. In this case, the random walk is particularly persistent in the sense that the rises and descents are no longer independent. A renewal property may still hold but it is more tedious to expect in general. Let $\underline S$ and $\overline S$ be the double-infinite comb PRWs with respective transitions
\begin{multline}\label{tilde}
\underline{\alpha}_{n}^{\ttu}:= \sup\{q_{c}(\ttd) : c\in  \mathcal C_{\ttu^{n}\ttd}\},\quad \underline{\alpha}_{n}^{\ttd}= \inf\{q_{c}(\ttu) : c\in  \mathcal C_{\ttd^{n}\ttu}\},\\
\mbox{and} \quad \overline{\alpha}_{n}^{\ttu}:= \inf\{q_{c}(\ttd) : c\in  \mathcal C_{\ttu^{n}\ttd}\},\quad \overline{\alpha}_{n}^{\ttd}:= \sup\{q_{c}(\ttu) : c\in  \mathcal C_{\ttd^{n}\ttu}\}.
\end{multline}

The proof of the following proposition is omitted and follows from comparison results exposed in \cite{PRWI}.

\begin{prop} \label{grafts}
Assume that $\underline S$ and $\overline S$ satisfy the same Assumption \ref{assglobal} with  $\alpha\in(0,1)$. To be more precise, we suppose that the associated  mean-drifts are equal and non-extremal but also that the corresponding random variables $\underline \tau_1^{\mathtt c}$ and $\overline \tau_1^{\mathtt c}$ are both in ${\rm D}(\alpha)$, the stable limit distribution being 
the same by taking the same normalization and centering terms. Then Theorem \ref{complement} holds for  $S^{\mathtt g}$.
\end{prop}

For instance, Theorem \ref{complement} holds when the (non-trivial) grafts are some finite trees in finite number  such that the attached Bernoulli distributions are non-degenerated and induce running times in ${\rm D(\alpha)}$.

%%%%%%%%%%%%%%%%%%%%%%%%%%%%%%%%%%%%%%%%%%%%%%%%%%%%%%%%%%%%%%%%%%%%%%%%%%%%%%%%%%%%%%%%%%%%%%%%%%%%%%%%%%%%%%%%%%%%%%%%%%%%%%%%%%%%%%%%%%%%%%%%%%%%%%%%%%%%%%%%%%%%%%%%%%%%%%%%%%%%%%%%%%%%%%%%%%%%%%%%%%%%%%%%%%%%%%%%%%%%%%%%%%%%%%%%%%%%%%%%%%%%%%%%%%%%%%%%%%%%%%%%%%%%%%%%%%%%%%%%%%%%%%%%%%%%%%%%%%%%%%%%%%%%%%%%%%%%%%%%%%%%%%%%%%%%%%%%%%%%%%%%%%%%%%%%%%%%%%%%%%%%%%%%%%%%%%%%%%%%%%%%%%%%%%%%%%%%%%%%%%%%%%%%%%%%%%%%%%%%%%%%%%%%%%%%%%%%%%%%%%%%%%%%%%%%%%%%%%%%%%%%%%%%%%%%%%%%%%%%%%%%%%%%%%%%%%%%%%%%%%%%%%%%%%%%%%%%%%%%%%%%%%%%%%%%%%%%%%%%%%%%%%%%%%%%%%%%%%%%%%%%%%%%%%%%%%%%%%%%%%%%%%%%%%%%%%%%%%%%%%%%%%%%%%%%%%%%%%%%%%%%%%%%%%%%%%%%%%%%%%%%%%%%%%%%%%%%%%%%%%%%%%%%%%%%%%%%%%%%%%%%%%%%%%%%%%%%%%%%%%%%%%%%%%%%%%%%%%%%%%%%%%%%%%%%%%%%%%%%%%%%%%%%%%%%%%%%%%%%%%%%%%%%%%%%%%%%%%%%%%%%%%%%%%%%%%%%%%%%%%%%%%%%%%%%%%%%%%%%%%%%%%%%%%%%%%%%%%%%%%%%%%%%%%%%%%%%%%%%%%%%%%%%%%%%%%%%%%%%%%%%%%%%%%%%%%%%%%%%%%%%%%%%%%%%%%%%%%%%%%%%%%%%%%%%%%%%%%%%%%%%%%%%%%%%%%%%%%%%%%%%%%%%%%%%%%%%%%%%%%%%%%%%%%%%%%%%%%%%%%%%%%%%%%%%%%%%%%%%%%%%%%%%%%%%%%%%%%%%%%%%%%%%%%%%%%%%%%%%%%%%%%%%%%%%%%%%%%%%%%%%%%%%%%%%%%%%%%%%%%%%%%%%%%%%%%%%%%%%%%%%%%%%%%%%%%%%%%%%%%%%%%%%%%%%%%%%%%%%%%%%%%%%%%%%%%%%%%%%%%%%%%%%%%%%%%%%%%%%%%%%%%%%%%%%%%%%%%%%%%%%%%%%%%%%%%%%%%%%%%%%%%%%%%%%%%%%%%%%%%%%%%%%%%%%%%%%%%%%%%%%%%%%%%%%%%%%%%%%%%%%%%%%%%%%%%%%%%%%%%%%%%%%%%%%%%%%%%%%%%%%%%%%%%%%%%%%%%%%%%%%%%%%%%%%%%%%%%%%%%%%%%%%%%%%%%%%%%%%%%%%%%%%%%%%%%%%%%%%%%%%%%%%%%%%%%%%%%%%%%%%%%%%%%%%%%%%%%%%%%%%%%%%%%%%%%%%%%%

%%%% 28/09/16 ici le soir, revoir remarque 3.4

\section{Few comments and contributions}

\label{sectioncontri}

\setcounter{equation}{0}
After some clarifications on the choice of the normalizing functions -- see Section \ref{sectionslut} -- and the connections between the scaling limits associated with the piecewise constant or the linear interpolations in Theorem \ref{thmgene0}, we give some interpretations on various parameters and limits involved in Theorems \ref{scalingbrownian} and \ref{complement} -- see Section \ref{sectionthm}. Thereafter -- see Section \ref{sectext} -- we briefly explain how to extend our results to the extremal mean drift situations. Then -- see Section \ref{phase} --  we interpret the essential difference between the limit processes in Theorems \ref{scalingbrownian} and  \ref{complement} as a phase transition phenomenon with respect to the memory. Finally, the last Section \ref{sectCTRW} is devoted to the relation with a part of the literature on CTRWs and their connections with the L\'{e}vy walks or the DRRWs.

\subsection{Regular variations and Slutsky type arguments}

\label{sectionslut}

First of all, we shall see that the normalization function and those  involved in its construction are regularly varying. More precisely, we can show that they are characterized by the relations
\begin{equation}
\frac{a(u)}{(\Sigma\circ a(u))^2}\;\underset{u\to\infty}{\sim}\; u\quad\mbox{and}\quad \Theta\circ a\circ s(u)\, s(u) \;\underset{u\to\infty}{\sim}\; u.
\end{equation} 

Then, in virtue of Slutsky type arguments -- see \cite[Theorem 3.4, p. 55]{ResH} and \cite[Proposition 3.1., p. 57]{ResH} or \cite[Theorem 3.1., p. 27]{Billcvg} for instance -- together with classical properties about regularly varying functions -- see the reference \cite{BGT} -- it comes that in the definition of $\lambda(u)$, along with $\Sigma(t)$ and $\Theta(t)$, the functions $a(u)$ and $s(u)$ can replaced by any other equivalent function in a neighbourhood of infinity to get the same functional convergences.

In this spirit,  when the persistence times are both square integrable, one can replace $\Xi_{2}(u)$ in (\ref{scalingBM2000}) by the the standard deviation of $\tau_1^{\mathtt c}$ to  retrieve the CLT shown in \cite{peggy}. Also,  when $\mathbf d_{\scr T}<\infty$, one can replace $D(u)$  by $\mathbf d_{\scr T}$ and thus $\mathbf m_{\scr S}$ by $\mathbf d_{\scr S}$ in  (\ref{scalingcauchy2}). 

\subsection{About the three theorems}

\label{sectionthm}

First, we stress that each of the functional convergences in (\ref{scalingform}) implies the other one. Indeed, due to the tightness criterion \cite[Theorem 12.12.3., p. 426]{WW} and the form of the oscillation function for the $M_1$-metric,  the tightness of one induces the tightness of the other. Then, the identification  of the finite-dimensional distribution follows from $S_t-1\leq S_{\lfloor t\rfloor}\leq S_t+1$ and the Slutsky lemma. As a consequence, the general Theorem \ref{thmgene0} is a direct consequence of Theorems \ref{scalingbrownian} and \ref{complement}.

 Furthermore,  we deduce from (\ref{scalingcauchy2}) and (\ref{Du}) the following  interpretation of the so-called mean drift $\mathbf m_{\scr S}$ -- appearing here as a drift in probability -- since for $\alpha\in[1,2]$, the following holds
\begin{equation}
\lim_{n\to\infty}\frac{S_{n}}{n}\;\overset{\mathbb P}{=}\;\mathbf m_{\scr S}.
\end{equation}
Even if the latter convergence is a straightforward consequence of the almost sure convergence exposed in \cite{PRWI} under the assumption that the running times are both integrable, it is completely new otherwise.

To go further, if $\alpha=1/2$ and ${\mathbf m}_{\scr S}=0$ in Theorem \ref{complement} then one can note that the distribution $f_t(x)$ given in (\ref{density}) is nothing but the push forward image by $x\longmapsto 2tx-1$ of the  classical arcsine distribution, the law of the occupation time of the half-line, up to the time $t$, of a one-dimensional standard Brownian motion. 
%As a more recent reference we can  cite \cite{FKY}.
In full generality, it also appears as the limit of the mean sojourn time of some classical random walks, including a wide class of discrete time processes described by Lamperti in \cite{Lamperti!}. It is also explained in \cite{BPY,WataLampLaw} that it is -- up to an affine transformation -- the law of the occupation time of the half-line of a skew Bessel process of dimension $2-2\alpha$ with a skewness parameter $(1+{\mathbf m}_{\scr S})/2$. Moreover, it is shown in \cite[Corollary 4.2., p. 343]{PitYor} that  $f_1(x)$ can be represented as the distribution of the random variable   
\begin{equation}\label{arcsin}
\mathbf D_{\alpha}:=\frac{T_{\alpha}^{\ttu}(1)-T_{\alpha}^{\ttd}(1)}{T_{\alpha}^{\ttu}(1)+T_{\alpha}^{\ttd}(1)}.
\end{equation}
Hence, the marginal convergence at time $t=1$ in (\ref{arcsinescaling})
%\begin{equation}\label{scalingform-1}
%\lim_{n\to\infty}\frac{S_{n}}{n}\;\overset{\mathcal L}{=}\;\mathbf D_{\alpha},
%\end{equation}
can be interpreted as a kind of law of large number in distribution which extends the classical one in \cite{PRWI} -- the means in (\ref{drift-def2}) being replaced by some canonical positive $\alpha$-stable random variables. 

In addition, the terminology used for the so-called mean drift $\mathbf m_{\scr S}$ is a direct consequence of the latter considerations since in any case $\alpha\in]0,2]$ we get that
\begin{equation}\label{meandrift1}
\lim_{n\to\infty}\frac{S_n}{n}\;\overset{\rm L^{1}}{=}\;\mathbf m_{\scr S}. 	
\end{equation}

Besides -- among other considerations -- it turns out that $(\mathcal S_{\alpha}$, $\mathcal X_{\alpha}$, $\mathcal A_{\alpha}$) is  somehow the continuous time counterpart of ($S$, $X$, $A$) and the label process $\mathcal X_{\alpha}$ is in some sense a continuous time VLMC. We recall that the discrete age time process $A$ is  introduced in (\ref{agetimediscret}) as $S$ the corresponding PRW and the two-sided sequence of jumps $X$. Indeed -- as it is explained in \cite{peggy} -- the stochastic process $\{(X_{n},A_{n})\}_{n\geq 0}$ on   $\{\ttu,\ttd\}\times \mathbb Z^{+}$ is also a Markov process.

\subsection{On extremal situations}

\label{sectext}

We have exclude in Assumption \ref{assglobal} the extremal mean drifts -- say $\mathbf m_{\scr S}=1$ for the example -- situation that may arise when $\alpha=1$ and $\mathbf d_{\scr T}=\infty$ or when $\alpha\in(0,1)$. In some sense, this mean -- in a first approximation -- that the weight of the persistence time $\tau_1^\ttd$ in the scaling limits can be omitted.

As a matter of facts, we can obtain similar functional convergences as   (\ref{scalingcauchy2}) or (\ref{arcsinescaling}) but there toward  the null or the identity process respectively. To this end, it suffices to replace  $\Sigma(t)^{2}$  by  $V(t)$ defined in (\ref{sumtheta})  in the settings of $a(u)$ and the proofs follows the same lines. 

The additional conditions to get non trivial limits are more subtle and related to the negligible persistence time. Roughly speaking, when $\tau_1^\ttd$ belongs to ${\rm D}(\gamma)$ with $0<\gamma\leq \alpha<1$ -- hypothesis $(\#)$ -- or when it is a relatively stable distribution in the sense of \cite[Chap. 8.8, p. 372]{BGT} -- hypothesis $(\star)$  -- then the generic form of the functional convergence seems to be 
\begin{multline}
\left\{
-\frac{u^{1-\frac{\alpha}{\gamma}}}{\Xi(u)}\left(\frac{S_{\lfloor u \,t\rfloor}}{u}- t\right)\right\}_{t\geq 0} 
\;\xRightarrow[u\to\infty]{{\rm J}_1}\; 
\left\{{T}_{\gamma}^\ttd\circ N_\alpha^\ttu (t))\right\}_{t\geq 0}\quad(\#),\\
\quad\mbox{or}\quad 
\left\{
-\frac{u^{1-\alpha}}{\Xi(u)}\left(\frac{S_{u\,t}}{u}- t\right)\right\}_{t\geq 0} 
\;\xRightarrow[u\to\infty]{{\mathcal C}}\; 
\left\{N_\alpha^\ttu (t))\right\}_{t\geq 0}\quad(\star),
\end{multline}
where $\Xi(u)$ is slowly varying and $T_\gamma^\ttd$ is a $\gamma$-stable subordinator  independent -- contrary to the non-extremal situation -- of an $\alpha$-local time $N_\alpha^\ttu$. In other words the limit process is an uncoupled CTRW under the assumption $(\#)$.

Furthermore, assuming in place of the relative stability that $\tau_1^\ttd\in{\rm D(\gamma)}$ with $\gamma\in[1,2]$, the next term in the asymptotic expansion of $(\star)$ is -- heuristically -- of the  order of $S_\gamma^\ttd\circ N_\alpha^\ttu(t)/u^{\alpha/\gamma}$ where $S_\gamma^\ttd$ is a $\gamma$-stable L\'{e}vy process independent of $N_\alpha^\ttu$.

\subsection{Phase transition phenomenon on the memory}

\label{phase}

It follows from our results and Assumption \ref{assglobal} that the limit process and the proper normalizations of the walk given in (\ref{scalingform}) strongly depend on  the tail distributions  of the length of runs. In fact, we bring out phase transition phenomena concerning the speed of diffusivity and the memory of the limit process. 

When the tail distributions are sufficiently light -- that is $\alpha\in[1,2]$ in (\ref{cond10}) -- the limit process is Markovian -- we lose the long term memory. In addition, the sub-linear normalizing function $\lambda(u)$ will be $(1/\alpha)$-regularly varying  as the index of self-similarity of the limit process. 

When the tails are sufficiently heavy -- that is $\alpha\in(0,1)$ --  the limit process is no-longer Markovian. We keep an unbounded  memory encoded -- in addition to its position -- by the   \enquote{current} label $\mathcal X_\alpha(t)$ and the age time $\mathcal A_\alpha(t)$. Moreover, the limit process is somehow ballistic and it is self-similar of index $1$. 

%we shall see that the mean drift $\mathbf m_{\scr S}$ is equal to the classical almost sure drift $\mathbf d_{\scr S}$, except for the boundary case $\alpha=1$ when the persistence times are both of infinite mean. In this situation, it is then equal to the balance term $\mathbf b_{\scr S}$. Moreover, 
%and we call it an arcsine Lamperti  anomalous diffusion. The normalization function $\lambda(u)$ is this time exactly linear. 

%\subsection{Beyond and around CTRW}

\subsection{Around and beyond CTRWs}

\label{sectCTRW}

In full generality, anomalous diffusions -- as the lagging process $X_\alpha(t)$ in (\ref{anomalousdiff2}) -- appear as scaling limits of (possibly coupled) CTRWs named also renewal reward processes or L\'{e}vy walks following the context. Those are nothing but classical RW subordinated to a counting process. Their scaling limits and their finite-dimensional distributions has been widely investigated in \cite{Meer2,Meer1,Meer1cor,MeerStra,Straka} for example. Besides, closely related to our model, DRRW has been introduced in \cite{Mauldin1996} to model ocean surface wave fields. Those are nearest neighbourhood random walks on $\mathbb Z^{d}$ keeping their directions during random times $\tau$, independently and identically drawn after every change of directions.  We refer to \cite{Scalingdirectionaly1998,Rastegar:2012,Scalingdirectionaly2014} where are revealed diffusive and super-diffusive behaviours. These walks are intrinsically continuous and can be seen as a linear interpolation of a CTRW -- as the \say{true} L\'{e}vy walks studied in \cite{Magdziarz}.

Our model generalizes one dimensional DRRW since asymmetrical transition probabilities $\alpha_{n}^{\ttu}$ and $\alpha_{n}^{\ttd}$ lead in general to running times $\tau^{\ttu}_1$ and $\tau^{\ttd}_1$ with distinct distributions. Even if the mean drift $\mathbf m_{\scr S}$ represents the classical almost sure drift $\mathbf d_{\scr S}$ of the PRW when it is  well defined, it appears to be a drift in mean  when the persistence times are not integrable and thus also in full generality as it is pointed out by (\ref{meandrift1}). This relevant characteristic has not yet been revealed to our knowledge.

Furthermore, it seems that some results in \cite{Scalingdirectionaly1998,Scalingdirectionaly2014} are not entirely true when the persistence times are not integrable, and so it  must have some  misunderstanding in their proofs. We think more precisely to \cite[Theorem 1.3., p. 370]{Scalingdirectionaly1998} and equation (3.24), p. 3281 relative to the proof of \cite[Theorem 2.6., p. 3272]{Scalingdirectionaly2014}. In fact, with our settings, those results would imply that 
\begin{equation}\label{scalingform-2}
\frac{S_{n}}{n}\;\xRightarrow[n\to\infty]{\mathcal L}\;	Y_{\alpha}(1).
\end{equation}
However, the latter convergence can not be true since  $Y_{\alpha}(1)$ is the overshoot limit of the underlying CTRW and thus it is not compactly supported -- contrary to $X_\alpha(1)$. Moreover, in their results it misses the residual term insuring the continuity of the limit process. 

 As a matter of fact, this kind of subtle mistake -- corrected in \cite{Meer1cor} -- has already been made in\cite{Meer1}. In that paper the authors prove some powerful limit theorems concerning coupled CTRW. The outcomes of our interest are \cite[Theorem 3.1., p. 738]{Meer1} together with \cite[Theorem 3.4., p.744]{Meer1} upgraded with  the slight but significant modification in \cite{Meer1cor}.

Finally, we also  fill some gaps in the literature. In addition to only investigating symmetric DRRWs in \cite{Scalingdirectionaly1998,Scalingdirectionaly2014}, the authors have omitted to treat the Cauchy situation $\alpha=1$.  This case is more delicate as it is explained in Section \ref{sectionl\'{e}vy}. To be more precise, the choice of the suitable centering term is not clear and the counting process $N(t)$ associated with the time $\tau_1^\ttd+\tau_1^\ttu$ no longer satisfies a SLLN so that we need to revise the proof when $\alpha\in(1,2]$. 

Moreover, in the case when $\alpha\in (0,1)$, there is no available proof of a functional convergence (at best a marginal one) and the limit distribution is not explicit to our knowledge. In this paper, we give the explicit density of such marginal limit in (\ref{density}) and we adopt the following  complementary approaches, according to the desired aims, in Section \ref{sectionanomalous}:

\begin{enumerate}
\item[1)] We adapt the results in \cite[Section 5.]{Magdziarz} representing the marginal  densities of some L\'{e}vy walks  as the solution of pseudo partial differential equations. 
%We only identify the mean and prove the desired functional convergence with this method.
\item[2)] We make the link with the excursions theory of $\alpha$-stable subordinators presented for instance in \cite[section 4., pp. 341-343]{PitYor} and \cite[Theorem 1]{BPY}. 
\item[3)] We make the connection with the occupation time of wide class of stochastic processes modelled on those defined by Lamperti in \cite{Lamperti!} and including $\{X_nX_{n+1}\}_{n\geq 0}$.
\end{enumerate}

\section{Proof of Theorem \ref{scalingbrownian}}

\setcounter{equation}{0}

\label{sectionl\'{e}vy}

We begin with the case of a L\'{e}vy motion as a limit process. The proofs follow more or less  the same lines depending on the index of stability and therefore we only insist on the most difficult -- and never exposed -- situation when $\alpha=1$. The proof in the other cases will be only outline.

\subsection{Sketch of the proof and preliminaries}

Let $N$ be the continuous time counting process associated with $T$ -- the increasing discrete time RW. The latter is introduced and represented -- as $M$ the skeleton RW -- in Section \ref{Settings}. More precisely, this renewal process is defined by 
\begin{equation}
N(t):=\max\{n\geq 0 : T_{n}\leq t\}=\inf\{n\geq 0 : T_{n+1}>t\}.	
\end{equation} 

The main idea of the proofs when $\alpha\in(1,2]$ consists of taking advantage of the decomposition 
\begin{equation}\label{decompoini}
S_{\lfloor u\,t\rfloor}-{\mathbf m_{\scr S}\,ut}=:\left(M_{N(u\,t)}-{\mathbf m_{\scr S}}T_{N(u\,t)}\right)+{R}(ut)=:{C}_{N(u\,t)}+ {R}(u\,t),	
\end{equation}
where $\{{C}_{n}\}_{n\geq 0}$ denotes the RW  whose jumps are distributed as the random variable $\tau_{1}^{\mathtt c}$ given in  (\ref{cond10}) and $\{R(v)\}_{v\geq 0}$ is the corresponding residual continuous time process. Note that
\begin{equation}\label{boundresi}
|R(v)|\leq (1-\mathbf m_{\scr S})\tau_{N(v)+1}^{\ttu}+(1+\mathbf m_{\scr S})\tau_{N(v)+1}^{\ttd}.
\end{equation}
%Also, remember that $\mathbf m_{\scr S}=\mathbf d_{\scr S}$ when $1<\alpha\leq 2$. 
Actually, the proofs of  (\ref{scalingBM2000}) and (\ref{scalinggeneric}) is organized as follows.\\

{\bf Step 1.} Direct consequence of Assumption \ref{assglobal} and estimations (\ref{equi}) together with classical results on $\alpha$-stable distributions we observe that $C$ belongs to the domain of attraction of the $\alpha$-stable L\'{e}vy process $S_{\alpha,\beta}$ -- $B$ when $\alpha=2$ -- with the normalization function given by $a(u)$.\\

{\bf Step 2.} Applying the weak law of large numbers to $T$, we show that $\{N(ut)/s(u)\}_{t\geq 0}$ converge in probability towards the identity process $t\mapsto t$ as $u$ goes to infinity.\\

{\bf Step 3.} We prove that the normalization of the residual process $\{R(ut)/a(u)\}_{t\geq 0}$ converges  in probability towards the null process $t\mapsto 0$ for $u$ large enough.\\ 

Applying a classical continuous mapping argument as in \cite[Sect. 17., Chap. 3., pp. 143-150]{Billcvg} for instance -- or to be more explicit by using \cite[Theorem 13.2.2., p. 430, Chap. 13.]{WW} coupled with the continuous mapping \cite[Theorem 3.4.3., p. 86, Chap. 3.]{WW} -- we can deduce the required scaling limits with the normalizing function $\lambda(u)=a\circ s(u)$ and this ends the proof when $\alpha\in(1,2]$ assuming the three latter steps.

\subsection{Focus on the Cauchy situation}

Unfortunately, in the case when $\alpha=1$, the Step 1.\@ is not true -- with the corresponding limit $\mathfrak C$ -- if the persistence times are both integrable $\mathbf d_{\scr T}<\infty$. Besides, we also need to be careful about the Step 2.\@ if $\mathbf d_{\scr T}=\infty$ since the weak law of large numbers does not hold in that case. Therefore, we need to adapt the reasoning above according the two latter situations.\\

First, when the persistence times are both of infinite mean, we can see that the proof follows exactly the same steps as previously by noting that Step 2. remains valid. Indeed, in any case $T$ is relatively stable in the sense of \cite[Chap. 8.8, p. 372]{BGT} with the inverse of $u\mapsto s(u)$ as normalizing function.

\begin{lem}[Step 2.]\label{Step2}
\label{lemmecount} The following convergences hold in probability: for any $v\geq 0$, 
\begin{equation}
\sup_{0\leq t\leq v}\left|\frac{N( u\,t)}{s(u)}- t\right|\;\xrightarrow[u\to\infty]{\mathbb P}\;0.
\end{equation}
\end{lem}
 
 \begin{proof}
Using the switching identity we first remark that it suffices to show that
\begin{equation}\label{Break}
\sup_{0\leq t\leq v}\left|\frac{T_{\lfloor s(u)t\rfloor}}{u}- t\right|\;\xrightarrow[u\to\infty]{\mathbb P}\;0.
\end{equation}
%Indeed,  the switching identity implies that for any $v\geq 0$ and $\delta>0$,
%\begin{equation*}
%\sup_{0\leq t \leq v}\left(\frac{N(u\,t)}{s(u)}- t\right)>\delta
%\;\Longrightarrow\;
%\inf_{0\leq t \leq v+\delta}\left(\frac{T_{\lfloor s(u)t\rfloor}}{u}- t\right)\leq -\delta,
%\end{equation*}
%and if in addition $\delta<v$, it leads to
%\begin{equation*}
%\inf_{0\leq t \leq v}\left(\frac{N(u\,t)}{s(u)}- t\right)<-\delta
%\;\Longrightarrow\;
%\sup_{0\leq t \leq v-\delta}\left(\frac{T_{\lfloor  s(u)t\rfloor}}{u}- t\right)\geq \delta.
%\end{equation*}

When $\mathbf d_{\scr T}<\infty$ the convergence in (\ref{Break}) at any fixed time $t$ but without the supremum follows immediately from the weak law of large numbers since $s(u)$ can be chosen to be equal to $u/\mathbf d_{\scr T}$. 

When $\mathbf d_{\scr T}=\infty$, the same marginal convergence is a direct consequence of the considerations on the tail distribution $\tail(t)$ of $T$ in Lemma \ref{lemmeequi} and classical results on stable distributions -- especially \cite[p. 174]{Geluk} -- which imply the convergence in distribution as $u$ tends to infinity of 
\begin{equation}
\frac{T_{\lfloor u\,t\rfloor}}{a(u)}- \frac{\Theta\circ a(u)\,u}{a(u)} t.
\end{equation}

Thus, in both cases, it only remains to prove the local uniform convergence in probability. To this end, it is well-known for stable distributions that the convergence of one marginal distribution implies its functional counterpart. Here the same is true. One way to prove this is to use the characteristics of semi-martingales which are in addition processes with independent increments. More precisely, by using \cite[Theorem 2.52, Chap. VII.2, p. 409]{Jacod} and \cite[Theorem 3.4, Chap. VII.3, p. 414]{Jacod} we only need to prove that the characteristics $(b_u,\tilde c_u,\nu_u)$ relative to a truncation function $h$ of $t\mapsto T_{\lfloor s(u)t\rfloor}/u$ tends to $(t\mapsto t,0,0)$ as $u$ tends to infinity uniformly on compact subset. Following \cite[Theorem 3.11, Chap. II.3, p. 94]{Jacod} and \cite[Equation 3.18, Chap. II.3, p. 96]{Jacod} it turns out that
\begin{equation}\label{Char3}
b_{u}(t)=\mathbb E\left[h\left(\frac{\tau_1}{u}\right)\right]\lfloor u\,t\rfloor
,\quad 
\tilde c_{u}(t)=\mathbb V\left[h\left(\frac{\tau_1}{u}\right)\right]\lfloor u\,t\rfloor\quad
\mbox{and}\quad
\nu_u([0,t]\times g):=\mathbb E\left[g\left(\frac{\tau_1}{u}\right)\right]\lfloor ut\rfloor,
\end{equation}
where $\tau_1$ is a jump of $T$ and $g$ is any bounded continuous function bounded by $x\mapsto x^2$ in a neighbourhood of the origin. Then the punctual convergence of this triplet follows from the convergence for $t=1$. Furthermore, their uniform convergence on compact sets are obvious and we deduce the convergence in distribution of the associated processes in the Skorokhod space endowed with the $\rm J_{1}$-Borel $\sigma$-field. Since the convergence in distribution to a constant in metric spaces implies the convergence in probability, the lemma is proved.
\end{proof}

By contrast, when the persistence times are both integrable, we will see that we need to adjust the first step and more precisely the decomposition (\ref{decompoini}) involving the jump distribution $\tau_1^\mathtt c$ which is no longer well balanced. This dichotomy appears in the choice of the centering term in the classical convergence towards a symmetric Cauchy process. 

\begin{lem}[Step 1. when $\alpha=1$]\label{lemmecenter} Let $W$ be a random walk whose jump  $w_{1}$ belongs to the domain of attraction of a symmetric Cauchy distribution. Then the centering term can be chosen to be equal to zero, when $w_1$ is not integrable, and to the drift otherwise,  in such way that 
\begin{equation}\label{cauchy}
\left\{\frac{W_{\lfloor u\,t\rfloor}}{k(u)}\right\}_{t\geq 0}\quad\mbox{or}\quad  \left\{\frac{W_{\lfloor u\,t\rfloor}-\mathbb E[w_{1}]\,u\,t}{k(u)}\right\}_{n\geq 0}  \;\xRightarrow[u\to\infty]{{\rm J_{1}}}\; 
\left\{{\mathfrak C}(t)\right\}_{t\geq 0},
\end{equation}	
according to the situations. Here we  denote by $k(u)$ the suitable normalizing function so that the limit process is the symmetric Cauchy process defined previously. 
\end{lem}

\begin{proof}
First, we can show by using \cite[pp. 170-174]{Geluk} that a suitable centering term (in the numerator above) to get the marginal convergence at time $t=1$  above is
\begin{equation}
b(u):=u\left(\int_0^{k(u)}\mathbb P(+w_1>s)ds- \int_0^{k(u)}\mathbb P(-w_1>s)ds\right),
\end{equation}
which can be rewritten when $w_1$ is integrable as 
\begin{equation}
b(u):=u\left[\mathbb E[w_1]-\left(\int_{k(u)}^\infty\mathbb P(w_1>s)ds- \int_{k(u)}^\infty\mathbb P(-w_1>s)ds\right)\right].
\end{equation}
In fact, the convergence holds when $u$ runs through the integers in most of the papers but  since $b(u)$ and $k(u)$ are regularly varying (indexes $0$ and $1$ respectively) a Slutsky type argument implies the convergence for $u$ along the real numbers. Besides, it is well known that the functional convergence is a consequence of the marginal at $t=1$. 

To go further, we can see that for $t$ sufficiently large there exist $\lambda^{\pm}(t)$  such that
\begin{equation}\label{dehaan}
\int_{t}^{\infty}\mathbb P(\pm w_{1}> s)ds =\frac{1}{2} \int_{\lambda^{\pm}(t)}^{\infty} \mathbb P(|w_{1}|> s)ds\quad
\mbox{or}\quad
\int_{0}^{t}\mathbb P(\pm w_{1}> s) ds =\frac{1}{2} \int_{0}^{\lambda^{\pm}(t)} \mathbb P(|w_{1}|> s)ds,	
\end{equation}
depending on whether $w_1$ is integrable or not. Moreover, since the right tail and the left tail of $w_1$ are well balanced, standard results on regularly functions (in particular those of slowly variations) and their inverse allow us to see that $\lambda^{\pm}(t)/t$ tend to $1$ as $t$ goes to infinity. Furthermore, noting that the two-sided tail of $w_1$ is regularly varying of index $1$, the de Haan theory applies -- especially \cite[Theorem 3.7.3, pp. 162-163]{BGT} coupled with \cite[Theorem 3.1.16, p. 139]{BGT} -- and with (\ref{dehaan}) it implies that   
\begin{equation}
{\int_{u}^\infty\mathbb P(w_1>s)ds- \int_{u}^\infty\mathbb P(-w_1>s)ds}\quad{\mbox{or}\quad \int_0^{u}\mathbb P(w_1>s)ds- \int_0^{u}\mathbb P(-w_1>s)ds},
\end{equation}
depending on whether $w_1$ is integrable or not, is negligible with respect to ${u\,\mathbb P(|w_1|>u)}$ as $u$ goes to infinity. To conclude, it suffices to note that $u\,\mathbb P(|w_1|>k(u))$ is equivalent to $1$ in a neighbourhood of infinity and thus by a Slutsky argument we deduce convergences (\ref{cauchy}).
\end{proof}

Therefore, when the persistence times are both of infinite mean, one can see the first step above holds with a symmetric Cauchy process for limit since $\tau_1^{\mathtt c}$ is well balanced and the centering term is null. This is not the case when the running times are both integrable. To overcome those difficulties, we write in place of the decomposition (\ref{decompoini}) the more adapted relation
\begin{equation}
S_{\lfloor u\,t\rfloor}-{\mathbf m_{\scr S}\,u\,t}=:
\left(S_{\lfloor u\,t\rfloor}-{\mathbf b_{\scr S}\,u\,t}\right) - (\mathbf d_{\scr S}-\mathbf b_{\scr S})\,u\,t=:C_{N(u\,t)}^{\mathtt o}-(\mathbf d_{\scr S}-\mathbf b_{\scr S})\,u\,t + {R^{\mathtt o}}(u\,t).
\end{equation}
We recall that $\mathbf m_{\scr S}=\mathbf d_{\scr S}$ in this situation. Here $C^{\mathtt o}$ denotes the RW associated with the (well balanced) jump $\tau_{1}^{\mathtt o}:=(1-\mathbf b_{\scr S})\tau_{1}^{\ttu}-(1+\mathbf b_{\scr S})\tau_{1}^{\ttd}$ and ${R^{\mathtt o}}$ is a residual process satisfying the same upper bound as the previous $R$ replacing $\mathbf m_{\scr S}$ by $\mathbf b_{\scr S}$. Note that $\tau_{1}^{\mathtt o}$ is in the domain of attraction of a symmetric Cauchy distribution -- a direct consequence of Lemma \ref{lemmeequi} -- and has for expectation 
\begin{equation}
\mathbb E[\tau_{1}^{\mathtt o}]=\mathbf d_{\scr T}(\mathbf d_{\scr S}-\mathbf b_{\scr S}).
\end{equation}
As a consequence -- since we can take $s(u):=u/\mathbf d_{\scr T}$ -- we deduce from Lemmas \ref{Step2} and \ref{lemmecenter} and that 
\begin{equation}
 \left\{\frac{C_{N(u\,t)}^{\mathtt o}-(\mathbf d_{\scr S}-\mathbf b_{\scr S})\,u\,t}{a\circ s(u)}\right\}_{n\geq 0}  \;\xRightarrow[u\to\infty]{{\rm J_{1}}}\; 
\left\{{\mathfrak C}(t)\right\}_{t\geq 0}
\end{equation}

In any cases, to conclude, we need to show that the normalization of the residual term is negligible and it is done in the proof of the following

%This one satisfies a similar upper bound as (\ref{boundresi}) and more exactly 
%\begin{equation}\label{boundresi2}
%|R^{\mathtt o}(v)|\leq (1-\mathbf b_{\scr S})\tau_{N(v)+1}^{\ttu}+(1+\mathbf b_{\scr S})\tau_{N(v)+1}^{\ttd}.
%\end{equation}

\begin{lem}[Step 3.]
The following convergences hold in probability: for any $v\geq 0$, 
\begin{equation}
\sup_{0\leq t\leq v}\left|\frac{R( u\,t)}{a\circ s(u)}\right|\;\xrightarrow[u\to\infty]{\mathbb P}\;0\quad \mbox{and}\quad \sup_{0\leq t\leq v}\left|\frac{R^{\mathtt o}( u\,t)}{a\circ s(u)}\right|\;\xrightarrow[u\to\infty]{\mathbb P}\;0.
\end{equation}
\end{lem}

\begin{proof}
First note that whatever $\alpha\in[1,2]$ it suffices to prove that for any $\ell\in\{\ttu,\ttd\}$,
\begin{equation}\label{residtau}
\sup_{0\leq t\leq v}\frac{\tau_{N(u\,t)+1}^{\ell}}{a\circ s(u)}\xrightarrow[u\to\infty]{\mathbb P}0.
\end{equation}	

It follows from Lemma \ref{lemmecount} above that for any $\delta>0$, there exists $s\geq 0$ such that for all $u\geq s$, $\mathbb P(\Omega_{u,\delta})\geq 1-\delta,$	with
\begin{equation}
\Omega_{u,\delta}:=\left\{\sup_{0\leq t \leq v}\left|\frac{N(ut)}{s(u)}- t\right|\leq \delta\right\}.
%=\bigcap_{0\leq t\leq v} \Big\{ N(ut)\in [s(u)(t-\delta),s(u)(t+\delta)]\Big\}.	
\end{equation}
Then, one can see for any $\varepsilon>0$ that
\begin{equation}\label{enfin1}
\mathbb P\left(\sup_{0\leq t \leq v} \frac{\tau_{N(u\,t)+1}^{\ell}}{a\circ s(u)}>\ep\;,\;\Omega_{u,\delta}\right)\leq
1-\left[1-\mathbb P(\tau_1^\ell>\ep \,a\circ s(u))\right]^{2\delta s(u)}.
\end{equation}
Furthermore, we obtain from the regular variations of the involved functions and Lemma \ref{lemmeequi} the existence of a positive constant $k$ depending only on $\mathbf m_{\scr S} $, $\mathbf b_{\scr S}$ and $\ell$ such that 
\begin{equation}\label{enfin3}
r \,\mathbb P(\tau_1^{\ell} >\ep a(r))
\underset{r\to\infty}{\sim}
k\left(\frac{2-\alpha}{\alpha}\right)\ep^{-\alpha}.	
\end{equation}
Therefore, we deduce from (\ref{enfin1}) and (\ref{enfin3})  that
\begin{equation}
\limsup_{u\to\infty}\;\mathbb P\left(\sup_{0\leq t \leq v} \frac{\tau_{N(u\,t)+1}^{\ell}}{a\circ s(u)}>\varepsilon\right)\leq \delta + \left[1-\exp\left(-2\delta k\left(\frac{2-\alpha}{\alpha}\right) \ep^{-\alpha}\right)\right],
\end{equation}
for any $\delta>0$, which achieves the proof of this lemma.
\end{proof}

Since the first step is obvious when $\alpha\in(1,2]$, this completes the proof of the L\'{e}vy situations, excepted Lemma \ref{cauchylemme} whose proof is given below.  
\hfill $\qed$

\begin{proof}[Proof of Lemma \ref{cauchylemme}]
First, one can see from Lemma \ref{lemmeequi} and standard results on regularly varying functions  that $\Xi_1(u):=a(u)\tail\circ a(u)$. Besides, the second asymptotic holds with $D(u):=\Theta\circ a\circ s(u)$. Then Karamata's Theorem  \cite[Proposition 1.5.9a., p. 26]{BGT} implies that $u\tail(u)$ is negligible with respect to $\theta(u)$ as $u$ goes to infinity and we deduce the lemma.
\end{proof}

%%%% ICI 29/09/16 

\section{Proof of Theorem \ref{complement}}

\setcounter{equation}{0}

\label{sectionanomalous}

We treat in this section the anomalous case, {\it i.e.\@} $\alpha\in(0,1)$. The proof is divided into four parts. First, we show the functional invariance principle -- Section \ref{functionalscalelim}. Then -- Section \ref{governingequation} -- we describe the marginal distributions of the process before we investigate -- Section \ref{markovprop} -- the Markov property given at the end of the theorem. Finally, we identify the marginal distributions in Section \ref{excursion} by the mean of the excursion theory but also -- Section \ref{Lampertipro} -- by extending the structure of Lamperti processes introduced in \cite{Lamperti!}.  

%%29/09/16 ICI

\subsection{Functional scaling limits}

\label{functionalscalelim}

First note that the tightness in (\ref{arcsinescaling}) is obvious since the modulus of continuity of $\{S_{u\,t}/u\}_{t\geq 0}$ is equal to $1$ almost surely. Hence, we only need to show the convergence of the finite-dimensional marginal distributions. To this end, we shall adapt the results in \cite{Magdziarz} for L\'{e}vy walks. In that paper, the authors consider L\'{e}vy walks of the form 
\begin{equation}
\mathfrak S(t):=\sum_{k=1}^{N(t)} \Lambda_k\, J_k + (t-T_{N(t)})\Lambda_{N(t)+1},
\end{equation}
with 
\begin{equation}
T_n:=\sum_{k=1}^n J_k\quad\mbox{and}\quad N(t):=\max\{k\geq 1 : T_k\leq t\}.
\end{equation}
Here the random moving times $J_k$ and the jumps $\Lambda_k\,J_k$ are {\it i.i.d.\@} -- the   $\Lambda_k$ being {\it i.i.d.\@} and independent of the jumping times and of unit size. Here, one can interpret $S$ -- the PRW -- as  such L\'{e}vy walk for which the random moving times are given  by $\tau_{2n-1}^\ttd$ or $\tau_{2n}^\ttu$ alternatively   and corresponding random jumps $-\tau_{2n-1}^\ttd$ or $\tau_{2n}^\ttu$. Hence, the directions are deterministic and equal to $-1$ or $1$ alternatively. 
%Also, one can see that the PRW is, up to some random variable, upper and lower bounded by  a L\'{e}vy walk having $\tau_n=\tau_n^\ttu+\tau_n^\ttd$ and $Y_n=\tau_n^\ttu-\tau_n^\ttd$ for random moving times and associated jumps. 

Even if their assumptions do not fit to this model, we shall see that their results extend easily to our situation. Indeed, the central  convergence \cite[Theorem 3.4., pp. 4023-4024]{Magdziarz} becomes in our situation
\begin{equation}
\left(\frac{M_{n}}{a(n)},\frac{T_{n}}{a(n)}\right)\;\xRightarrow[n\to\infty]{\mathcal L}\;(S_{\alpha}(1),T_{\alpha}(1)),
\end{equation}  
where $T_\alpha$ and $S_\alpha$ are the coupled $\alpha$-stable L\'{e}vy processes introduced to define the arcsine Lamperti anomalous diffusion. Thereafter, it is not difficult to check that the same continuous mapping and topological arguments used in the proof of \cite[Theorem 4.11., p. 4032]{Magdziarz} and its Corollary 4.14., p. 4033, hold to obtain the functional convergence (\ref{arcsinescaling}).

\subsection{About the governing equation}

\label{governingequation}

We observe that $\mathcal S_\alpha$ is the same scaling limit of a \say{true} L\'{e}vy walk -- in the sense of \cite{Magdziarz} -- whose random moving times and jumps are respectively given by  $\xi_n \tau_n^\ttu+(1-\xi_n) \tau_n^\ttd$ and $\xi_n \tau_n^\ttu-(1-\xi_n) \tau_n^\ttd$ with  {\it i.i.d.\@} random directions $2\xi_n-1$ independent of the running times and distributed as a Rademacher distribution of parameter $(1+\mathbf b_{\scr S})/2$ denoted in the following by $\mathcal X(d\ell)$.

\begin{rem}
One can see the latter L\'{e}vy walk as randomized version of the PRW which is somehow more convenient to study as it is already appeared for the recurrence and transience criteria in \cite{PRWI}.  
\end{rem}

It is also shown in \cite[Theorem 5.6., pp. 4036-4037]{Magdziarz}  that the density distribution $f_t(x)$ satisfies -- in a weak sense -- the fractional partial differential equation  (\ref{goveq}) since it can be write as
\begin{equation}
\int_{\mathcal A}\left(\frac{\partial }{\partial t}+\ell\,\frac{\partial }{\partial t}\right)^\alpha  f_t(x)\,\mathcal X(d\ell)=\frac{1}{\Gamma(1-\alpha)}\frac{1}{t^\alpha}\int_{\mathcal A}\delta_{\ell \,t}(dx)\,\mathcal X(d\ell).
\end{equation}
Besides, it follows from \cite[Section 5., p. 4035]{Magdziarz} that $f_t(x)$ can be expressed as   
\begin{equation}
f_t(x)=\frac{1}{\Gamma(1-\alpha)}\int_{\mathcal A}\int_0^t \frac{u_\alpha(x-\ell\,(t-s),s)}{(1-s)^\alpha}\,ds\,\mathcal X(d\ell),
\end{equation}
where $u_\alpha(x,t)$ denotes the $0$-potential  density of  $(S_\alpha,T_\alpha)$ -- the occupation time density -- and it comes (\ref{potential}). Furthermore, the Fourier-Laplace Transform (FLT) of $u_\alpha(x,t)$ can be computed explicitly  and it leads to the FLT of $\mathcal S_\alpha(1)$ given in equation (5.5) of  \cite{Magdziarz} -- equal here to
\begin{equation}
\frac{(s-iy)^{\alpha-1}+(s+iy)^{\alpha-1}}{\left(\frac{1+\mathbf b_{\scr S}}{2}\right)(s-iy)^\alpha+\left(\frac{1-\mathbf b_{\scr S}}{2}\right)(s+iy)^\alpha}.
\end{equation}
Taking the derivative  at the origin with respect to the spacial variable   we can see that  the mean distribution of $f_1(x)$ is equal to $\mathbf m_{\scr S}$. However, we have not been able to invert directly this FLT to get the expression (\ref{density}). 

\begin{rem}
By using the results in \cite{MeerStra} it may be possible to provide the same representations of  the finite-dimensional marginal distributions.
\end{rem}

\subsection{Markov property and continuous time VLMC}

\label{markovprop}

Here we show that $(\mathcal X_\alpha,\mathcal A_\alpha)$ is Markovian. Let $\mathbb P_0^{\alpha}$ be the distribution of such stochastic process. First, from the regenerative property of  a stable subordinator we get that $\{(\mathcal X_{\alpha}(t+s),\mathcal A_{\alpha}(t+s))\}_{t\geq 0}$ is distributed  as $\mathbb P_0^{\alpha}$ for any $s\in\mathcal R_\alpha$. Furthermore -- by using \cite[Lemma 1.10, p. 15]{Bertoin} -- one can see that
\begin{equation}
\mathbb P(\mathcal H_{\alpha}(t)\in dh\mid\mathcal A_{\alpha}(t)=a)=\frac{\alpha\,a^{\alpha}}{(a+h)^{\alpha+1}}\,dh,	
\end{equation}
for any $t>0$ and we denote this homogeneous kernel by $N(a;dh)$. Then we introduce for any $\ell\in\{-1,1\}$ and $a\geq 0$ the distribution $\mathbb H^{\alpha}_{(\ell,a)}$ of the process on $\{-1,1\}\times[0,\infty)$ equal to $t\mapsto (\ell,a+t)$ up the random time $H_a$ distributed as $N(a;dh)$ and whatever elsewhere. Therefore, by construction and the regenerative property, it turns out  that the distribution of the stochastic process $\{(\mathcal X_{\alpha}(t+s),\mathcal A_{\alpha}(t+s))\}_{t\geq 0}$ is -- almost surely -- the push-forward image of $\mathbb H^{\alpha}_{(\mathcal X_\alpha(s),\mathcal A_\alpha(s))}\otimes \mathbb P_0^{\alpha}$ under the gluing map $G$ defined by
\begin{equation}
G(w_1,w_2):= w_1(t)\mathds 1_{\{t< H_a\}}+w_2(t)\mathds 1_{\{t>H_a\}},
\end{equation}
The Markov property is then a simple consequence of this representation. 

As a consequence, one  can see $\mathcal X_\alpha$ as a continuous time VLMC. Besides, it is well known -- see \cite[Chap. 8.6]{BGT} for instance -- that the distribution of the couple formed with the age time and the remaining time $(\mathcal A_\alpha(t),\mathcal H_\alpha(t))$ can be computed explicitly and some scaling limits of this couple leads to other generalized arcsine distributions. Moreover, one can easily be convinced using  \cite[Theorem (3.2), p. 506]{Maisonneuve} or the results in  \cite{Fitz,Kaspi}) that  this process admit the invariant (infinite) measure given by   
\begin{equation}
\int_{\mathcal A}\int_{0}^\infty G_\ast (\mathbb H^{\alpha}_{(\ell,a)}\otimes \mathbb P_0^{\alpha})\frac{1}{a^{\alpha+1}}\,da \, \mathcal X(d\ell),
\end{equation}
where $G_\ast\mu$ is the push-forward image of a distribution $\mu$ by the map $G$. One can note that this measure is not a probability and that to be compared with the result \cite{peggy} stating that a the double-comb PRW admits an invariant probability measure if and only if the persistence times are both integrable.

\subsection{Arcsine Lamperti density}

To conclude the proof of Theorem \ref{complement} it remains to 
we compute explicitly the density expression. The first proof uses the theory of excursion whereas the second one generalizes the concept of Lamperti and is a more general result.

\subsubsection{Via the excursion theory}

\label{excursion}

We deduce from the integral representation (\ref{limitintegral}) that \cite[section 4., pp. 341-343]{PitYor} holds and more precisely that \cite[Theorem 1]{BPY}  apply to our situation. Hence, we get the density expression, making in addition the link with the distribution of the occupation time in the positive half-line of a skew Bessel process of dimension $2-2\alpha$ and skewness parameter $(1+\mathbf b_{\scr S})/2$.

\subsubsection{A second class of Lamperti processes}

\label{Lampertipro}

Lamperti investigate in \cite{Lamperti!} the general question of the distribution of the  occupation time of a set $A$ for some stochastic process $\{\mathbf X_n\}_{n\geq 0}$ on a state space $E$. Regarding the dynamic, it is surprisingly  only assumed that $E$ can be divided into two sets -- the aforementioned $A$ and an other one $B$ -- plus a recurrent state $\sigma$ in such a way that for any $n\geq 1$, if $\mathbf X_{n-1}\in A$ and $\mathbf X_{n+1}\in B$ or vice versa, then necessarily  $\mathbf X_n=\sigma$. 

Thereafter, given such process starting from $\sigma$, Lamperti consider $N_n$ the occupation time of $A$ up to the time $n$ -- the state $\sigma$ being counted when the process comes from $A$ -- and he introduces $F(x)$ the generating function associated with the recurrence time of $\sigma$. This decomposition of the state space and how to count the occupation time is resumed by the weighted diagram in the left-hand side of Figure \ref{lampertistructure}. The following Theorem is stated and proved in its paper.

\begin{theo}
The mean occupation time $\{N_n/n\}_{n\geq 1}$ converges in distribution as $n$ goes to infinity  to a non-degenerate limit if and only if there exist $\alpha$ and $p$ in $(0,1)$ such that
\begin{equation}
\lim_{n\to\infty}\mathbb E[N_n/n]=p\quad\mbox{and}\quad \lim_{x\uparrow 1} (1-x)F^\prime(x)/(1-F(x))=\alpha.
\end{equation}
Besides --in that case -- the limit distribution has for density on $(0,1)$ the function 
\begin{equation}
\frac{\sin(\pi\alpha)}{\pi}\frac{t^{\alpha-1}(1-t)^{\alpha-1}}{rt^{2\alpha}+2\cos(\pi\alpha)t^\alpha(1-t)^{\alpha}+r^{-1}(1-t)^{2\alpha}},
\end{equation}
with $r:=(1-p)/p$.
\end{theo}

\begin{figure*}[!b]
\centering
\includegraphics[width=124mm]{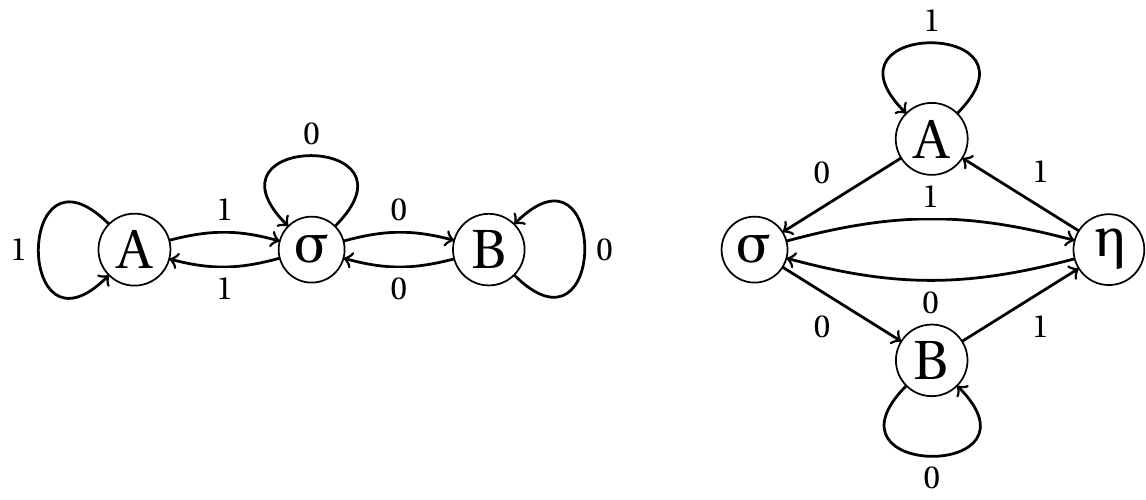}
\caption{Original and current Lamperti decomposition of the state space}
\label{lampertistructure}
\end{figure*} 

The proof is based on a precise asymptotic estimate of the moments associated with $N_n$ as $n$ goes to infinity -- via suitable generating functions and their regular variations -- in order to identify the limit of the Stieljes transform of $N_n/n$.

In our case, it would seem natural to see $A$ as the rises of $S$ so that it can be write more or less as $S_n=2N_n-n$. Unfortunately, we have not been able to fit this situation to the latter Lamperti dynamic. As a matter of fact, the communication diagram and the way to compute the suitable occupation time $N_n$ can be given as in the right hand side of Figure \ref{lampertistructure}.

To be more precise, if we consider the 2-letters  process $\mathbf X_{n}:=X_{n}X_{n+1}$ with the decomposition of  $\{\ttu,\ttd\}\times \{\ttu,\ttd\}$ given by $A:=\{\ttu\ttu\}$, $B:=\{\ttd\ttd\}$, $\sigma:=\ttu\ttd$ and $\eta:=\ttd\ttu$ and if  we denote by $N_n$ the occupation time of $A$ up to the time $n$ -- the state $\eta$ being counted or not according to well chosen weights -- the PRW satisfies for any $n\geq 1$ the relation $S_{n}=2N_{n-1}-n$. Note also that $\sigma$ and $\eta$ are recurrent states and since the random times spend in $A$ and $B$ are distributed as $\tau_1^\ttu$ and $\tau_1^\ttd$ respectively  the recurrence times of $\sigma$ and $\eta$ are both equal in law to $\tau_1=\tau_1^\ttu+\tau_1^\ttd$. 

To get the limit distribution we adapt and follow a part of the proof in \cite{Lamperti!}. To this end, we introduce $p_{n,k}$ the probability -- starting from $\sigma$ -- that $N_{n}=k$ and we denote  by $u_{n}$ and $d_n$ the probability that $\tau_{1}^{\ttu}=n$ and $\tau_{1}^{\ttd}=n$ respectively. Thereafter,  by conditioning with respect to the complete three events  $\{\tau_1\leq n\}\sqcup\{\tau_1\geq n+1,\tau^{\ttd}_1\leq n\}\sqcup\{\tau^\ttd_1\geq n+1\}$ we obtain
\begin{equation}
p_{n,k}=\sum_{\substack{m+l\leq  n\\m\leq k}}	d_{l}\,u_{m}\,p_{n-(m+l),k-m}+\sum_{l\leq n}d_{l}\tail_{\ttu}(n-l)\delta_{k,n-l+1}+\tail_{\ttd}(n)\delta_{k,0},
\end{equation}
with $\delta_{i,j}$ is $1$ or $0$ according to $i=j$ or not. Then, let $F_\ell(z)$ be the generating function of $\tau^\ell_1$ and $T_\ell(z)$ be the one associated with its tail distribution. it follows that the double generating function associated with the latter difference equation satisfies 
\begin{equation}
P(x,y):=\sum_{n,k\geq 0} p_{n,k}\,x^{n}\,y^{k}=\frac{F_{\ttd}(y)T_{\ttu}(xy)y+T_{\ttd}(x)}{1-F_\ttd(x)F_\ttu(xy)},	
\end{equation}
From our stability assumption and  classical results on regular variations we get that  there exists a slowly varying function $L(z)$ such that
\begin{equation}
1-F_\ttu(z)=(1-z)T_\ttu(z)\;\underset{z\uparrow 1}{\sim}\; \left(\frac{1+\mathbf b_{\scr S}}{2}\right)L\left(\frac{1}{1-z}\right)(1-z)^\alpha,
\end{equation}
the same asymptotic and equality being also true replacing $\ttu$ by $\ttd$ and $({1+\mathbf b_{\scr S})}/{2}$ by $({1-\mathbf b_{\scr S})}/{2}$. Then we can check that for any $\lambda>0$,
\begin{equation}
\lim_{x\uparrow 1}(1-x)P\left(x,e^{-\lambda(1-x)}\right)=	\frac{(1+\lambda)^{\alpha-1}+\mathbf r_{\scr S}}{(1+\lambda)^{\alpha}+\mathbf r_{\scr S}},
\end{equation}
and thereafter the proof follows exactly the same lines as \cite{Lamperti!}. We obtain the arcsine Lamperti density. Note that it may be possible to state and prove a general theorem for such Lamperti processes.\\

This completes the proof of the anomalous situation. 
\hfill $\qed$

\section{Equivalent characterizations of Assumption \ref{assglobal}}

\setcounter{equation}{0}

\label{appendix}

In the continuity of (\ref{def-tail}) and (\ref{truncatedmeansum}) we can consider $\tail_{\mathtt c}(t)$ and ${V}_{\mathtt c}(t)$ the two-sided distribution tail and the truncated second moment of the central random variable $\tau_{1}^{\mathtt c}$ given in (\ref{cond10}). In the following, we denote by ${\rm RV}(\lambda)$ the set of regularly varying functions of index $\lambda\in\mathbb R$, also named the set of slowly varying functions $\rm SV$ if $\lambda=0$. It is well known that  hypothesis (\ref{cond10})  is equivalent to  
\begin{equation}\label{equip}
{V}_{\mathtt c}(t)\in {\rm RV}(2-\alpha),\quad\mbox{or equivalently when $\alpha\neq 2$,}\quad 	\tail_{\mathtt c}(t)\in{\rm RV}(-\alpha),
\end{equation}
with an additional tail balance criterion when $\alpha\neq 2$: there exists $\beta\in[-1,1]$ with
\begin{equation}\label{balancetauc0}
\lim_{t\to\infty}\frac{\mathbb P(\tau^{\mathtt c}_1> t)-\mathbb P(\tau^{\mathtt c}_1< -t)}{\mathbb P(\tau^{\mathtt c}_1>t)+\mathbb P(\tau^{\mathtt c}_1<-t)}=\beta.
\end{equation}
Remark from \cite[Theorem 2, Chap. VIII.9, p. 283]{Feller} that
\begin{equation}
\lim_{t\to\infty}\frac{t^{2}\tail_{\mathtt c}(t)}{V_{\mathtt c}(t)}=\frac{2-\alpha}{ \alpha},	
\end{equation}
To go further, introduce
\begin{equation}\label{sumtheta}
\tail(t):=\tail_{\ttu}(t)+\tail_{\ttd}(t)
\quad\mbox{and}\quad	
V(t):=V_{\ttu}(t)+V_{\ttd}(t).
\end{equation}

The two following lemmas gives another possible interpretations and choices of condition (\ref{cond10}) and of the normalizing functions. Their proofs are straightforward computations using only classical results on regularly varying functions which can be found in \cite{BGT} for instance, they are omitted.

\begin{lem}\label{lemmeequi}
Assuming the existence of the  mean drift $\mathbf m_{\scr S}\in(-1,1)$, the hypothesis  (\ref{cond10}) is equivalent when $\alpha\neq 2$ to  $Y_{1}=\tau_{1}^{\ttu}-\tau_{1}^{\ttd}\in{\rm D}(\alpha)$ but also to
\begin{equation}\label{tailsum0}
\Big[V(t)\in{\rm RV}(2-\alpha)\quad\mbox{or}\quad \tail(t)\in{\rm RV}(-\alpha)\Big]\quad\mbox{and}\quad \lim_{t\to\infty}\frac{\tail_{\ttu}(t)-\tail_{\ttd}(t)}{\tail_{\ttu}(t)+\tail_{\ttd}(t)}=\mathbf b_{\scr S}.
\end{equation}
In that case, the functions $\tail(t)$ and $V(t)$ are respectively the tail distribution and the truncated second moment of $Y_{1}$ but also of $\tau_{1}$. Besides, one has 
\begin{equation}
\tail_{\mathtt c}(t)\;\underset{t\to\infty}{\sim}\;\left[(1-\mathbf m_{\scr S})^{\alpha}\left(\frac{1+\mathbf b_{\scr S}}{2}\right)+(1+\mathbf m_{\scr S})^{\alpha}\left(\frac{1-\mathbf b_{\scr S}}{2}\right)\right]\tail(t),
\end{equation}
and thus the same asymptotic between $V_{\mathtt c}(t)$ and $V(t)$. Moreover, the balance term $\beta$ in (\ref{balancetauc0}) satisfies
\begin{equation}\label{balancetauc00}
\beta=\frac{(1-\mathbf m_{\scr S})^{\alpha}(1+\mathbf b_{\scr S})-(1+\mathbf m_{\scr S})^{\alpha}(1-\mathbf b_{\scr S})}{(1-\mathbf m_{\scr S})^{\alpha}(1+\mathbf b_{\scr S})+(1+\mathbf m_{\scr S})^{\alpha}(1-\mathbf b_{\scr S})}.
\end{equation}
Finally, when $\alpha=2$, 
\begin{equation}\label{equi}
\mathbb V[\tau^{\mathtt c}_{1}\mathds 1_{\{|\tau^{\mathtt c}_1|\leq t\}}]
\;\underset{t\to\infty}{\sim}\;
{V}_{\mathtt c}(t)\;\underset{t\to\infty}{\sim}\;
\Sigma(t)^{2}.	
\end{equation}
\end{lem}

Hence, when $\alpha\neq 2$, Assumption \ref{assglobal} means that at least one of the persistence times belongs to the domain of attraction of a stable distribution,  the tail distribution of the other one being in some sense comparable with the first one. However, in view of \cite[Theorem 4.5., p. 790]{ShimGauss}, we can state

\begin{rem} When $\alpha=2$, it is possible for all linear combination of $\tau_{1}^{\ttu}$ and $\tau_{1}^{\ttd}$ distinct of  $\tau^{\mathtt c}_{1}$ to not belong to the domain of attraction of a standard normal distribution. 
\end{rem}

As a consequence, it turns out that $a(u)$, $s(u)$ and thus $\lambda(u)$ -- together with  $\Sigma(t)^{2}$ and $\Theta(t)$ inside their definitions --  are regularly varying functions.  In particular, we can see that the sub-linear normalizing function $\lambda(u)$ is  asymptotically linear if and only if $\alpha\in(0,1)$. in any case, there exists a slowly varying function $\ell(u)$ such that
\begin{equation*}
\lambda(u)\;\underset{u\to\infty}{\sim}\;\ell(u)\,u^{\frac{1}{\alpha \vee  1}}.
\end{equation*}

%~ \begin{proof}[Proof of Lemma \ref{bound}]
%~ These bounds are the consequence of the integration by parts formula (named the Abel's transformation) and the fact that the increments satisfy,
%~ \begin{equation*}
%~ \int_{n}^{n+1} x^{p} d\mu(x)  \;\underset{n\to\infty}{\sim}\; \frac{1}{n^{q-p}}   \int_{n}^{n+1} x^{q} d\mu(x).
%~ \end{equation*}
%~ \end{proof} 

The second lemma, coupled with the first one, will be useful -- among other considerations -- to discriminate whether the persistence times are integrable or not and thus to identify $\mathbf m_{\scr S}$. Again, the proof is omitted. Given two real functions $f(t)$ and $g(t)$ defined on a neighbourhood of infinity, we set
$f(t){\asymp} g(t)$ when there exists $c>0$ such that for $t$ sufficiently large $c^{-1} g(t)\leq f(t)\leq c g(t)$.

\begin{lem}\label{lemmebound}
Let $\mu$ be a positive measure on $[0,\infty)$ and for any $p\geq 0$ and $t\geq 0$,
\begin{equation*}
M_{p}(t):=\int_{[0,t]} x^{p}\mu(dx)\quad\mbox{and}\quad T_{p}(t):=\int_{(t,\infty)} x^{p} \mu(dx).
\end{equation*}
Then for any $q>p\geq 0$,
\begin{equation*}
M_{p}(t) \;\underset{}{\asymp}\; 
\int_{[1,t]}\frac{M_{q}(u)}{u^{q-p+1}} du + \frac{M_{q}(t)}{t^{q-p}}
\quad\mbox{and}\quad
T_{p}(t) \;\underset{}{\asymp}\; 
\int_{(t,\infty)}\frac{M_{q}(u)}{u^{q-p+1}} du + \frac{M_{q}(t)}{t^{q-p}}.
\end{equation*}
\end{lem}

Therefore, the persistence times are both integrable $\mathbf d_{\scr T}<\infty$ when $\alpha\in(1,2]$ and in that case $\mathbf m_{\scr S}=\mathbf d_{\scr S}$. Besides, they are both of infinite mean when $\alpha\in(0,1)$ and then $\mathbf m_{\scr S}=\mathbf b_{\scr S}$, the two latter situations being possible when $\alpha=1$.

\bibliography{biblio-vnby}
\bibliographystyle{unsrturl}
\end{document}